\documentclass[a4paper,11pt,leqno]{article}

\usepackage[latin1]{inputenc}
\usepackage[T1]{fontenc} 
\usepackage[english]{babel}
\usepackage{verbatim}
\usepackage{calligra}
\usepackage{enumerate}
\usepackage{dsfont}
\usepackage{amsfonts}
\usepackage{amsmath}
\usepackage{amsthm}
\usepackage{amssymb}
\usepackage{mathtools}
\usepackage{mathrsfs}
\usepackage{color}
\usepackage{graphicx}
\usepackage{bm}
\usepackage{srcltx}

\usepackage{tikz}
\usepackage[retainorgcmds]{IEEEtrantools}

\usepackage{graphicx}		
\usepackage[hypcap=false]{caption}

\DeclareMathOperator*{\wtend}{\rightharpoonup}

\theoremstyle{definition}
\newtheorem{defin}{Definition}[section]
\newtheorem{ex}[defin]{Example}
\newtheorem{rmk}[defin]{Remark}

\theoremstyle{plane}
\newtheorem{thm}[defin]{Theorem}
\newtheorem{prop}[defin]{Proposition}
\newtheorem{cor}[defin]{Corollary}
\newtheorem{lemma}[defin]{Lemma}

\newcommand{\tbf}{\textbf}

\newcommand{\tsl}{\textsl}

\newcommand{\mbb}{\mathbb}

\newcommand{\mc}{\mathcal}
\newcommand{\mf}{\mathfrak}
\newcommand{\mds}{\mathds}

\newcommand{\veps}{\varepsilon}

\newcommand{\s}{\sigma}
\renewcommand{\t}{\tau}

\newcommand{\de}{\delta}
\renewcommand{\o}{\omega}

\newcommand{\R}{\mathbb{R}}

\newcommand{\N}{\mathbb{N}}

\newcommand{\T}{\mathbb{T}}

\renewcommand{\div}{{\rm div}\,}

\newcommand{\curl}{{\rm curl}\,}

\newcommand{\supp}{{\rm supp}\,}

\newcommand{\dx}{ \, {\rm d} x}
\newcommand{\dt}{ \, {\rm d} t}

\newcommand{\loc}{{\rm loc}}

\allowdisplaybreaks

\def\d{\partial}
\def\div{{\rm div}\,}
\newcommand{\dd}{{\rm d}}

\newcommand{\fra}[1]{\textcolor{blue}{#1}}

\begin{document}

\title{\textsc{\Large{\textbf{Yudovich theory under geometric regularity for density-dependent incompressible fluids}}}}

\author{\normalsize \textsl{Francesco Fanelli}$\,^{1,2,3}$ 
\vspace{.5cm} \\
\footnotesize{$\,^{1}\;$ \textsc{BCAM -- Basque Center for Applied Mathematics}} \\ 
{\footnotesize Alameda de Mazarredo 14, E-48009 Bilbao, Basque Country, SPAIN} \vspace{.2cm} \\
\footnotesize{$\,^{2}\;$ \textsc{Ikerbasque -- Basque Foundation for Science}} \\  
{\footnotesize Plaza Euskadi 5, E-48009 Bilbao, Basque Country, SPAIN} \vspace{.2cm} \\
\footnotesize{$\,^{3}\;$ \textsc{Universit\'e Claude Bernard Lyon 1}, {\it Institut Camille Jordan -- UMR 5208}} \\ 
{\footnotesize 43 blvd. du 11 novembre 1918, F-69622 Villeurbanne cedex, FRANCE} \vspace{.3cm} \\
%
%
\footnotesize{Email address: \ttfamily{ffanelli@bcamath.org}}
\vspace{.2cm}
}

\date\today

\maketitle

\subsubsection*{Abstract}
{\footnotesize 

This paper focuses on the study of the density-dependent incompressible Euler equations in space dimension
$d=2$, for low regularity (\textsl{i.e.} non-Lipschitz) initial data satisfying assumptions
in spirit of the celebrated Yudovich theory for the classical homogeneous Euler equations.


We show that, under an \textsl{a priori} control of a non-linear geometric quantity, namely the directional derivative
$\d_Xu$ of the fluid velocity $u$ along the vector field $X:=\nabla^\perp\rho$, where $\rho$ is the fluid density,
low regularity solutions \textsl{\`a la Yudovich} can be constructed also in the non-homogeneous setting.
More precisely, we prove the following facts:
\begin{enumerate}[(i)]
 \item \emph{stability}: given a sequence of smooth approximate solutions enjoying a uniform control on the above mentioned geometric quantity,
then (up to an extraction) that sequence converges to a Yudovich-type solution of the density-dependent incompressible Euler system;
\item \emph{uniqueness}: there exists at most one Yudovich-type solution of the density-dependent incompressible Euler equations
such that $\d_Xu$ remains finite; besides, this statement improves previous uniqueness results for regular solutions,
inasmuch as it requires less smoothness on the initial data.
\end{enumerate}


}

\paragraph*{\small 2020 Mathematics Subject Classification:}{\footnotesize 35Q31 
(primary);
35R05, 
76B03, 
35A02. 
(secondary).}

\paragraph*{\small Keywords: }{\footnotesize incompressible Euler equations; density variations; Yudovich theory; geometric regularity;
stability and convergence; uniqueness.
}

%
%

\section{Introduction} \label{s:intro}

The incompressible Euler equations are one of the most iconic mathematical models arising in fluid mechanics. They describe the evolution
of the velocity field $u$ of an incompressible ideal fluid: in absence of external forces (for simplicity) and after denoting by $\Pi$ the pressure field
of the fluid, they take the form
\begin{equation} \label{eq:hom-E}
\left\{\begin{array}{l}
        \d_tu\,+\,(u\cdot\nabla)u\,+\,\nabla\Pi\,=\,0 \\[1ex]
        \div u\,=\,0\,.
       \end{array}
\right.
\end{equation}
Ideal fluid, in this context, means that the fluid is assumed to be inviscid (no friction forces act
between molecules of the fluid) and homogeneous, \tsl{i.e.} with constant density.

In this paper, we are interested in the generalisation of the above system in presence of variations of the density of the fluid. If we introduce a positive scalar
function $\rho\geq 0$ to describe the evolution of the fluid density, assuming that friction forces are still negligible,
the resulting system, known as the \emph{non-homogeneous incompressible Euler system}, becomes
\begin{equation} \label{eq:dd-E}
\left\{\begin{array}{l}
\d_t\rho\,+\,u\cdot\nabla\rho\,=\,0 \\[1ex]
\rho\,\d_tu\,+\,\rho\,(u\cdot\nabla) u\,+\,\nabla\Pi\,=\,0 \\[1ex]
\div u\,=\,0\,.
       \end{array}
\right.
\end{equation}
All the (scalar and vector) fields $\rho$, $u$ and $\Pi$ are assumed to be functions of time $t$ and space $x$. As we are interested in
the initial value problem associated to \eqref{eq:dd-E}, we set the equations for $(t,x)\in\R_+\times \Omega$, where, for the time being, $\Omega$ is a domain
of $\R^d$, with $d\geq2$.
We supplement the equations with the initial datum 
\begin{equation} \label{eq:in-datum}
 \big(\rho,u\big)_{|t=0}\,=\,\big(\rho_0,u_0\big)\,,\qquad\qquad \mbox{ with }\qquad 
 \div u_0\,=\,0\,.
\end{equation}
In addition, throughout this work we assume \emph{absence of vacuum}. This means that there exist two positive constants $0<\rho_*\leq \rho^*$ such that
\begin{equation} \label{eq:vacuum}
 0<\rho_*\leq\rho_0\leq\rho^*\,.
\end{equation}

System \eqref{eq:dd-E} retains a strong similarity with its homogeneous ancestor \eqref{eq:hom-E}, at least under assumption \eqref{eq:vacuum}.
For instance, both systems can be seen (neglecting non-local effects introduced by the pressure term) as a quasi-linear system
of transport equations.
Despite these premises, from a purely mathematical standpoint, there is a huge gap between the developed theory for the classical
incompressible Euler equations \eqref{eq:hom-E} and the results established so far for their non-homogeneous counterpart \eqref{eq:dd-E}.
Such a gap is particularly evident in the two-dimensional case: when $d=2$, celebrated results like global well-posedness and Yudovich theory of
weak solutions, well-known for equations \eqref{eq:hom-E}, are still largely open in presence of density variations.

The goal of the present paper is precisely to address the question of the validity of a \emph{weak solutions theory}  \tsl{\`a la Yudovich} for the non-homogeneous
incompressible Euler equations \eqref{eq:dd-E}. By this, we mean solutions, weak at the level of the vorticity formulation, such that the vorticity possesses only
$L^p$-type bounds, with no additional regularity; in particular, discontinuous vorticities are allowed and the regularity of the velocity field remains below
the Lipschitz threshold.
Thus, troughout this work we will set $d=2$ and consider equations \eqref{eq:dd-E} on
\[
 \R_+\times\Omega\,,\qquad\qquad \mbox{ with }\qquad \Omega\,=\,\R^2\,,
\]
even though the case of the two-dimensional torus $\Omega\,=\,\T^2$ can be dealt with as well, with minor modifications.

Let us immediately clarify that we are \emph{not} able to prove existence and uniqueness of solutions in a framework \tsl{\`a la Yudovich}. However, our study
highlights the crucial role played in this matter by a non-linear geometric quantity, recently introduced in \cite{F_2025}, which
consists of the directional derivative $\d_Xu$ of the velocity field $u$ along the direction $X:=\nabla^\perp\rho$ tangent to the level curves of the density.
Assuming \tsl{a priori} a control on such quantity, we are able to perform a stability analysis of smooth approximate solutions towards Yudovich-type solutions
to \eqref{eq:dd-E}, as well as to establish uniqueness of those solutions for which, in addition, $\d_Xu$ remains finite.
In passing, we notice that the uniqueness statement improves existing uniqueness results for regular solutions, inasmuch as it holds true under much less stringent
assumptions.

Before going into the core of the matter, let us briefly revisit the classical theory for the constant density case.

\subsection{Yudovich theory for the homogeneous Euler equations} \label{ss:yud}

Owing to the divergence-free condition, the incompressible Euler equations \eqref{eq:hom-E} can be equivalently formulated in terms of the vorticity
of the fluid, which is represented by the quantity $\Omega\,=\,Du-\nabla u$, where $Du$ is the Jacobian matrix of $u$ and $\nabla u$ its transpose.
In the case of space dimension $d=2$, the vorticity matrix can be identified with the scalar function
\begin{equation} \label{def:omega}
 \o\,:=\,\curl u\,=\,\d_1u^2\,-\,\d_2u^1\,.
\end{equation}
When computing the evolution equation for $\o$, a key cancellation (specific of the case $d=2$) occurs in the non-linear term, and the vorticity equation
reduces to a simple transport equation:
\begin{equation} \label{eq:hom-vort}
 \d_t\o\,+\,u\cdot\nabla\o\,=\,0\,.
\end{equation}
Notice that the velocity field $u$ can be recovered in terms of $\o$ by solving the Biot-Savart law
\begin{equation} \label{eq:BS-law}
 u\,=\,-\,\nabla^\perp(-\Delta)^{-1}\o\,=\,\frac{1}{2\pi}\,\int_{\R^2}\frac{1}{|x-y|^2}\,(x-y)^\perp\,\o(y)\,\dd y\,,
\end{equation}
where, given any vector $v\in\R^2$, $v\,=\,\big(v^1,v^2\big)$, we have set $v^\perp\,=\,\big(-v^2,v^1\big)$. 

Equation \eqref{eq:hom-vort} endows the Euler system with an infinite number of conservation laws, as (at least formally) all the $L^p$ norms of $\o$
are preserved by the evolution, for any $p\in[1,+\infty]$:
\begin{equation} \label{est:vort-hom}
\forall\,p\in[1,+\infty]\,,\quad \forall\,t\geq0\,,\qquad\qquad \left\|\o(t)\right\|_{L^p}\,\leq\,\left\|\o_0\right\|_{L^p}\,.
\end{equation}
In passing, we recall that this is the key to get global well-posedness for the $2$-D Euler equations
(see \tsl{e.g.} classical textbooks like \cite{Marc-Pulv}, \cite{Maj-Bert} and \cite{BCD} for a review of the known results).

In 1963, Yudovich \cite{Yud_1963} used the simple form of equation \eqref{eq:hom-vort} to set up a weak solutions theory to the incompressible
Euler equations in $2$-D, guaranteeing both existence and \emph{uniqueness} properties. Yudovich theorem can be stated in the following form
(see \tsl{e.g.} Theorems 8.1 and 8.2 of \cite{Maj-Bert}).

\begin{thm}
Given any  initial vorticity $\o_0\in L^1(\R^2)\cap L^\infty(\R^2)$, there exists a unique weak solution $\o$ to system \eqref{eq:hom-vort}-\eqref{eq:BS-law},
defined globally in time and such that $\o\in L^\infty\big(\R_+;L^1(\R^2)\cap L^\infty(\R^2)\big)$.
\end{thm}

Let us make a few comments about the previous statement, also in view of a comparison with our approach and main results.

First of all, let us highlight the two main properties of Yudovich theory: global in time existence, and uniqueness of solutions.
Uniqueness strongly relies on the assumption $\o_0\in L^\infty$, and on the property \eqref{est:vort-hom}.
As far as only existence is concerned, however, the integrability assumptions on $\o$ may be somehow relaxed
to the condition $\o_0\in L^{p_1}\cap L^{p_2}$, for some $1\leq p_1 <2<p_2\leq +\infty$, which is enough to define (\tsl{via}
the Biot-Savart law) $u$ as a $L^\infty$ function.

Next, we observe that, from the property $\o\in L^\infty\big(\R_+;L^1(\R^2)\cap L^\infty(\R^2)\big)$ and
Calder\'on-Zygmund theory, one can see that $u$ actually belongs to
$L^\infty\big(\R_+;W^{1,p}(\R^2)\big)$, for any $2<p<+\infty$. At the same time, by potential theory estimates applied to \eqref{eq:BS-law},
one can show that $u$ is log-Lipschitz continuous with respect to the space variable (but fails to be Lipschitz, in general), which is still
a sufficient condition to define a unique flow.

In addition, from the above mentioned properties and from the velocity equation in \eqref{eq:hom-E}, in turn one can recover
suitable information for the quantity $\d_tu$ and, by elliptic estimates, 
for the pressure term $\nabla\Pi$. 
Thus, Yudovich theory is a theory of weak solutions at the level of the vorticity equation, but
the constructed solutions  are actually \emph{strong} solutions\footnote{By \emph{strong solutions},
we mean here solutions for which all the terms in the equations are defined a.e. on $\R_+\times\Omega$, $\Omega$ being the spatial domain where the equations
are set.} of the incompressible Euler equations at the level of the velocity formulation \eqref{eq:hom-E}.
In passing, let us mention an important difference between the classical case \eqref{eq:hom-E} and the case of variable density \eqref{eq:dd-E}.
In the homogeneous framework, by solving
the equations in the vorticity form one can completely neglect the analysis of the pressure term.
This is no more the case for non-homogeneous fluids, for which a pressure term appears also in the equation for the vorticity.
In order to study its regularity, we will need to resort to its elliptic equation and argue precisely as just sketched above for \eqref{eq:hom-E}.

\medbreak
Of course, the literature concerning the study of the classical incompressible Euler equations is very vast, even when focusing
on the case of weak solutions in two space dimensions. Reviewing all those results is out of the scopes of this introduction.

Here, let us only mention that many studies have been devoted to relax the hypothesis $\o_0\in L^\infty(\R^2)$ still getting uniqueness of solutions.
Yudovich himself, in \cite{Yud_1995}, did so in bounded domains of $\R^d$ ($d=2$ or $3$, even though the existence of his class of solutions
in $3$-D remains unknown at present), by requiring a growth $\|\o_0\|_{L^p}\,\lesssim\,\theta(p)$, with $\theta(p)\longrightarrow+\infty$
for $p\to+\infty$ at a certain rate (for instance, the choice $\theta(p)=\log p$ is suitable). Then Vishik \cite{Vish}
extended that result to a class of borderline Besov spaces.
More recently, Bernicot and Keraani \cite{Bern-K} obtained global well-posedness in the logarithmic
space ${\rm LBMO}(\R^2)$, which is a subclass of ${\rm BMO}(\R^2)$ yet strictly larger than $L^\infty(\R^2)$. Further improvements to
this result can be found in \cite{Bern-Hm, C-M-Z, Crippa-Stef}.

If one dismisses the uniqueness requirement and only focuses on existence of solutions, more results are available (see 
\tsl{e.g.} to \cite{DiP-M, Chae, Tan}).
On the ill-posedness side, instead, a recent breakthrough by Vishik \cite{V-1, V-2}
proves the existence of infinitely many solutions for the \emph{forced} $2$-D incompressible Euler equations in vorticity form, for initial vorticities
$\o_0$ belonging only to $L^1(\R^2)\cap L^p(\R^2)$, with $p<+\infty$. About this issue, we refer also to the monograph \cite{A-B-C-DL-G-J-K},
where the authors revisit Vishik's proof and propose a slightly different approach.
We refrain from quoting here a large number of contributions about various forms of ill-posedness in incompressible fluid mechanics,
as far from the scopes of the present work.

Finally, we mention that the validity of the Yudovich theorem (that is, existence and uniqueness for bounded initial vorticities)
in bounded domains with corners was addressed in \cite{L-Miot-W}.
Related to the dynamics in bounded domains, we mention that several studies  \cite{Co-Wu_95, Co-Wu_96, Co-D-E} have addressed the problem of
the inviscid limit in the Yudovich class.

\medbreak
In the next subsection, we are going to draw a parallel between the homogeneous Euler equations and their non-homogeneous version \eqref{eq:dd-E},
and explain why the Yudovich theory cannot be directly applied to the latter. In doing this, we will present an overview of our basic working assumptions;
their precise formulation, together with the statement of the main results of this work, can be found in Section \ref{s:results}.

\subsection{The non-homogeneous case: motivations of the hypotheses} \label{ss:nonhom}
In order to establish a theory \tsl{\`a la Yudovich} for the density-dependent incompressible Euler equations, it is natural to consider
the equation for the vorticity $\o$, defined as in \eqref{def:omega}. We now present the vorticity equation in the non-homogeneous
framework and discuss its main consequences.


\paragraph*{The vorticity equation, and consequences.}
Taking advantage of the non-vacuum condition \eqref{eq:vacuum}, which can be propagated also at later times, 
we can compute an equation for the vorticity $\o$ in the non-homogeneous setting:
\begin{equation} \label{eq:vort}
 \d_t\o\,+\,u\cdot\nabla\o\,-\,\frac{1}{\rho^2}\,\nabla^\perp\rho\cdot\nabla\Pi\,=\,0\,,
\end{equation}
where, according with the notation introduced above, we have set $\nabla^\perp\,=\,\big(-\d_2,\d_1\big)$.
In particular, from equation \eqref{eq:vort} we see that, if we want to propagate some $L^p$ norm of the vorticity directly by its equation,
we cannot avoid to estimate suitable Lebesgue norms of the density gradient $\nabla\rho$ and of the pressure term $\nabla\Pi$.
At the same time, because of the additional density--pressure term, here we have no hope to get simple \tsl{a priori} estimates as in \eqref{est:vort-hom}.

All this marks a severe difference with respect to the homogeneous case, and main difficulties arise as a consequence.
In order to circumvent the complexity introduced by the vorticity equation \eqref{eq:vort},
we adopt a new approach to propagate regularity for $\o$, inspired by our recent work \cite{F_2025}. The idea is to rather consider
the ``vorticity'' of the momentum $m\,:=\,\rho\,u$, namely the function
\[
 \eta\,:=\,\curl m \,=\,\d_1m^2\,-\,\d_2m^1\,.
\]
Observe that, by direct computations, one has
\begin{equation} \label{intro_eq:vort-eta}
 \eta\,=\,\rho\,\o\,+\,u\cdot\nabla^\perp\rho\,,
\end{equation}
so $L^p$ bounds for $\o$ can be recovered from similar bounds for $\eta$, provided suitable controls are available for both $u$ and $\nabla\rho$.
Thus, within this approach a preliminary information on $u$ is needed to treat $\o$: this represents a change of perspective
with respect to the classical Yudovich theory, in which all the estimates available on $u$ are derived from the Biot-Savart law \eqref{eq:BS-law}.
At the same time, we see that $\eta$ satisfies the transport equation
\begin{equation*}
 \d_t\eta\,+\,u\cdot\nabla\eta\,=\,\d_Xu\cdot u\,, 
\end{equation*}
where we have defined the vector field $X\,:=\,\nabla^\perp\rho$ and we have used the notation $\d_Xu\,=\,(X\cdot\nabla)u$.
In view of the previous equation and formula \eqref{intro_eq:vort-eta} for $\o$, we remark that
there is an interplay between the Lebesgue bounds for $\eta$, $u$ and $\nabla\rho$ (or equivalently $X$).
A key point, pushed forward by the analysis of \cite{F_2025}, is that the quantity $\d_Xu$ carries a non-trivial \emph{geometric} information,
hence it has to be considered as a unique term and must not be splitted when performing $L^p$ estimates. 
From this perspective, a natural choice 
is to require a bound for $\d_Xu$ in $L^\infty$ and estimate $\eta$ and $u$ in the same $L^p$ space.

The $L^\infty$ bound on the quantity $\d_Xu$ is one of the crucial points of all our theory, and we want to explain it in detail.
For the sake of clarity, we postpone this discussion at the end of this part and focus now on less delicate issues.
To begin with, let us delve into the analysis of $u$.

Pursuing the previous approach, based on the use of the momentum $m$ and its vorticity $\eta$, we observe that
Lebesgue bounds for $u$ can be derived from corresponding bounds for $m$. Now, thanks to the property $\div u=0$,
the Leray-Helmholtz decomposition of $m$ writes
\begin{equation} \label{intro_eq:m}
 m\,=\,-\,\nabla^\perp(-\Delta)^{-1}\eta\,-\,\nabla(-\Delta)^{-1}\big(u\cdot\nabla\rho\big)\,.
\end{equation}
By Sobolev embeddings and Calder\'on-Zygmund theory, we infer that $L^p$ bounds of $u$, when $1<p<+\infty$, would require a control on the
$L^q$ norm of $u\cdot\nabla\rho$ and $\eta$, for some $q<\min\{2,p\}$.
Now, we anticipate that, in our framework, we have $\nabla\rho\in L^\infty$ (again, this condition will be better explained below).
Thus, complications arise when estimating $\|u\|_{L^p}$ if $p<+\infty$. Instead, this strategy is suitable to bound $u$ in $L^\infty$,
as shown in \cite{F_2025}.
Indeed, we can bound the $L^\infty$ norm of $m\,=\,\rho\,u$, thus of $u$, by its $W^{1,p}$ norm, for some finite $p>2$; an application of Calder\'on-Zygmund theory 
then makes the $L^p$ norm of $u$ appear, which however can be controlled by interpolation:
\[
 \|u\|_{L^p}\,\lesssim\,\|u\|_{L^{k}}^\theta\,\|u\|^{1-\theta}_{L^\infty}\,,\qquad\qquad \mbox{ for a suitable }\quad \theta\in\,]0,1[\,.
\]
This, however, requires to fix a low integrability index $k$ such that $k<p$. The natural candidate is $k=2$, because this choice allows us to
take advantage of the basic kinetic energy balance law associated to equations \eqref{eq:dd-E}, which implies, thanks to \eqref{eq:vacuum},
that $\|u(t)\|_{L^2}\leq \|u_0\|_{L^2}$ for any time $t\geq0$ for which the solution is defined. This choice will have also an impact on the proof of uniqueness,
simplifying the arguments used in the classical theory (see \tsl{e.g.} Chapter 8 of \cite{Maj-Bert}).

\medbreak
Let us summarise the main points of the previous discussion.
First of all, we will assume the initial velocity field $u_0$ to be of finite energy, that is $u_0\in L^2$. This is an additional assumption with respect to the
classical Yudovich theory, which then requires to weaken the conditions on the initial vorticity. It fact,
imposing $\o_0\in L^{p_0}$, for some finite $p_0>2$, is enough for the stability analysis; for ensuring uniqueness, instead, we need to require
moreover $\o_0\in L^\infty$, as in the classical work by Yudovich.
With those assumptions at hand, regularity/integrability of $u$ and $\omega$ will be propagated by using $\eta$ and the vector field
$X=\nabla^\perp\rho$. As mentioned above, the analysis of $\eta$ (and, as we will see, also the one of $X$) will make
the geometric quantity $\d_Xu$ come into play.

Before commenting on the evolution of $X$ and $\d_Xu$, we want to spend a few words on the analysis of the pressure term $\nabla\Pi$.

\paragraph*{The pressure gradient.}
In the non-homogeneous setting, the pressure term $\nabla\Pi$ cannot be neglected in the analysis.
There are a couple of reasons for that.
Firstly, despite our argument will often move from the momentum equation to the vorticity equation and \tsl{viceversa}, our ultimate goal
is to construct weak solutions to the vorticity equation, namely equation \eqref{eq:vort}: this requires to identify the pressure gradient in the analysis.
Secondly, the uniqueness argument will be based (similarly to Yudovich's argument) on $L^2$ stability estimates performed at the level of
the momentum equation. However, because of the presence of the non-homogeneity $\rho$, the stability estimates will invoke suitable bounds
precisely for the pressure term.

This having been pointed out, we observe that, for density-dependent fluids, $\nabla\Pi$ solves the variable coefficients elliptic problem
\begin{equation} \label{intro_eq:pressure}
 -\,\div\left(\frac{1}{\rho}\,\nabla\Pi\right)\,=\,\div\Big((u\cdot\nabla)u\Big)\,.
\end{equation}
A convenient framework to solve this equation independently of the size of the non-homogeneity $\rho$ is the $L^2$ setting, see \tsl{e.g.} \cite{A-T}.
Thus one looks for conditions able to guarantee that $(u\cdot\nabla)u$ belongs to $L^2$.
Notice that the assumption $u_0\in L^2$  is not really useful here, as, in any case, we do not dispose of a $L^\infty$ bound for $\nabla u$. Instead,
the property $(u\cdot\nabla)u\in L^2$ can be derived from different considerations, which are only slightly more involved than in the homogeneous case,
in which one can use the fact that $u\in L^\infty$ and $\o\in L^1\cap L^\infty$, thus $\nabla u\in L^2$.

In passing, we observe that this is also consistent with previous results about well-posedness in the class of strong solutions for equations \eqref{eq:dd-E},
where infinite energy solutions can be considered while still guaranteeing the property $\nabla\Pi\in L^2$ (see \tsl{e.g.} \cite{D:F, F_2012} in this respect).

\paragraph*{The gradient of the density.}
Formulas \eqref{intro_eq:vort-eta} and \eqref{intro_eq:m} demand a bound on suitable Lebesgue norms of the gradient of the density.
As anticipated above, we will propagate integrability of $\nabla\rho$ by looking at the evolution of the vector field
$X\,=\,\nabla^\perp\rho$, following in this way another idea from \cite{F_2025}. The key observation is that $X$ 
is transported (in the sense of vector fields) by the flow associated to $u$.
Namely, $X$ satisfies the transport equation
\begin{equation} \label{intro_eq:X}
 \d_tX\,+\,(u\cdot\nabla)X\,=\,\d_Xu\,, 
\end{equation}
where, as before, we have defined $\d_Xu\,:=\,(X\cdot\nabla)u$.
Notice that the quantity $\d_Xu$ makes sense pointwise a.e. here;
however it could be defined even weakly, as $\d_Xu=\div(X\otimes u)$ owing to the fact that $\div X=0$.
Some facts about equation \eqref{intro_eq:X} must be pointed out.

To begin with, we mention that $L^p$ bounds on $X$ can be propagated only assuming $\d_Xu\,\in\,L^1_T(L^p)$.
The estimate can be closed splitting $\left\|\d_Xu\right\|_{L^p}\,\leq\,\left\|X\right\|_{L^p}\,\left\|\nabla u\right\|_{L^\infty}$, but the Lipschitz norm
of the velocity field is out of control in a theory of solutions \tsl{\`a la Yudovich}. Thus, as already remarked above, and accordingly with the analysis
of \cite{F_2025}, we will consider the geometric quantity $\d_Xu$ as a whole. 

Next, we want to justify our choice of taking $p=+\infty$ in the previous argument. In fact,
in light of \eqref{intro_eq:vort-eta}, we notice that having $X\in L^\infty$ is absolutely fundamental in order to get uniqueness,
for which one needs the property $\o\in L^\infty$ as well.
On the contrary, it is possible to see that the condition $X\in L^{p_1}$, for some $p_1>p_0$,
would be enough for performing the stability analysis. However, this would slightly complicate the computations, without
yielding any gain in the assumptions, as an \tsl{a priori} control on the geometric quantity $\d_Xu$ would be needed anyway
for propagating the $L^{p_1}$ bound for $X$.
This is why we prefer to directly take $p_1=+\infty$.


Motivated by the previous considerations, we will assume $\rho_0$ to belong to $W^{1,\infty}(\R^2)$, together with, of course, the bounds
imposed in \eqref{eq:vacuum}. Thus, in this framework, we are not able to consider irregular densities, for instance densities presenting
a jump across a (say) smooth interface.
However, no additional integrability at infinity for $\rho_0$ or
$\nabla \rho_0$ is required, nor any kind of closedness to some constant value, and we can handle arbitrarily large density variations.


\paragraph*{Geometric regularity.}
As it appears clear from the above discussion, disposing of an \tsl{a priori} bound on the $L^1_T(L^\infty)$ norm of $\d_Xu$ is crucial for the whole theory.

However, when trying to estimate $\d_Xu = (X\cdot\nabla)u$ in $L^1_T(L^\infty)$, one is led to face deep difficulties.
Many natural ways to do that in our context are destined to fail. For instance, one may write
\[
 \left\|\d_Xu\right\|_{L^\infty}\,\leq\,\|X\|_{L^\infty}\,\left\|\nabla u\right\|_{L^\infty}\,,
\]
but the quantity $\|\nabla u\|_{L^\infty}$ is out of control in Yudovich theory, as well as in our framework: the velocity field is \emph{not} Lipschitz,
but only log-Lipschitz continuous.
Similarly, the relation $\d_Xu\,=\,\d_Xm/\rho$ is out of use here: owing to \eqref{intro_eq:m},
it would lead to estimating other singular integral operators in $L^\infty$, which simply suffer of the same problem we have just encountered.
Another strategy one may use consists in computing an equation for $\d_Xu$. This is easy: using the fact that $\d_X\rho=0$ and
equation \eqref{intro_eq:X}, applying the operator $\d_X$ to the momentum equation yields
\[
 \rho\,\d_t\d_Xu\,+\,\rho\,(u\cdot\nabla)\d_Xu\,=\,-\,\d_X\nabla\Pi\,=\,-\,(X\cdot\nabla)\nabla\Pi\,.
\]
Now, computing the derivatives in equation \eqref{intro_eq:pressure}, we see that only an information on $\Delta\Pi$ is available,
implying that one would need to estimate yet another singular integral operator, namely $\nabla^2(-\Delta)^{-1}$, in the $L^\infty$ norm.

Observe that classical tools for getting $L^\infty$ bounds for singular integral operators require to either move to the space
of H\"older continuous functions, or to use the critical Besov space $B^0_{\infty,1}$. Pursuing this strategy, however, naturally leads us
to a Lipschitz bound for $u$, thus falling in the strong solutions setting (see \tsl{e.g.} \cite{D_2010, D:F, F_2012, Brav-F}).
At the same time, 
we see that the quantity $\d_Xu$ contains a geometric information
about only two derivatives of $u$ performed along the  (twisted) direction $X\,=\,\nabla^\perp\rho$.
Strictly speaking, for this quantity to be bounded, it is \emph{not necessary} to
dispose of a global $L^\infty$ bound on $\nabla u$. 
An explicit example of a function which possesses regularity along a fixed direction but which is not globally Lipschitz continuous can be easily produced.

\begin{ex} \label{ex:tang}
Let $f\,=\,\mds{1}_{D}$ be the characteristic function of a (say) smooth bounded domain $D\subset\R^2$.
Then $f$ is clearly not Lipschitz.
Now, let $g:\R^2\longrightarrow\R$ be a global parametrisation of the boundary $\d D$ of $D$, namely a function
such that $\d D\,=\,g^{-1}\big(\{0\}\big)$ and $|\nabla g(x)|>0$ for any $x\in\d D$. Define $Y\,:=\,\nabla^\perp g$. Then
$\d_Yf\equiv0$ and, more in general, $\d_Y^kf\equiv0$ for any $k\geq1$.
\end{ex}

The previous considerations lead us to treat the quantity $\d_Xu$ as a whole.
Our point of view, undertaken also in the recent work \cite{F_2025},
is to some extent related to the notion of \emph{striated regularity}, first introduced
by Chemin in his pioneering works \cite{Ch_1991, Ch_1993} to study the regularity of vortex patches. 

At the same time, we need to impose \tsl{a priori} working assumptions directly on $\d_Xu$, because (as explained above)
we are not able to estimate this term from the equations.

\subsection{Overview of the main results} \label{ss:res-overv}
Our main contribution (formulated precisely in Section \ref{s:results}) is twofold. Firstly, we prove convergence of sequences
of smooth solutions $\big(\rho_n,u_n,\nabla\Pi_n\big)_{n\in\N}$ towards Yudovich-type solutions\footnote{Here and in what follows, the name of
Yudovich-type solutions refer to solutions constructed at the level of regularity described in Subsection \ref{ss:nonhom}.
See Subsection \ref{ss:results} for more details about this.}
to system \eqref{eq:dd-E}  for low regularity initial data (as described above) and under the \tsl{a priori} assumption that
\begin{equation} \label{intro_eq:g-ass}
\exists\,T>0 
\qquad \mbox{ such that }\qquad\qquad \sup_{n\in\N}\left\|\d_{X_n}u_n\right\|_{L^1_T(L^\infty)}\,<\,+\infty\,,
\end{equation}
where $X_n\,:=\,\nabla^\perp\rho_n$.
Notice that this is a geometric regularity requirement on $\big(\rho_n,u_n,\nabla\Pi_n\big)_{n\in\N}$.

It is fair to point out that a characterisation of low regularity initial data giving rise (after approximation) to a sequence  of smooth solutions for which
\eqref{intro_eq:g-ass} holds true, or even the construction of an explicit example of such data and solutions, has been so far elusive, and remains open at present.
Thus, our convergence result can be seen as a sort of \emph{conditional existence} result of Yudovich-type solutions to the density-dependent
incompressible Euler equations. It is worth remarking that the (conditional) existence of Yudovich-type solutions is obtained
precisely up to the time $T>0$ for which the geometric regularity condition \eqref{intro_eq:g-ass} holds.

The second main result of this work is to prove \emph{uniqueness} of Yudovich-type solutions 
to system \eqref{eq:dd-E}, which, besides the integrability and regularity properties mentioned above, satisfy in addition
the condition $\d_Xu\in L^1_T(L^\infty)$, with $X\,=\,\nabla^\perp\rho$.
Besides, our uniqueness statement improves previous uniqueness results stated for regular solutions to system \eqref{eq:dd-E} in a finite-energy framework,
see \tsl{e.g.} \cite{D_2006, D_2010, D:F, Brav-F}.

\subsection*{Organisation of the paper}

After this introduction, we now give an overview of the rest of the paper.

In Section \ref{s:results} we fix our basic assumptions and state our main results. For the sake of better clarity, we
try to separate the hypotheses which are needed for the stability and convergence of smooth approximate solutions
(result stated in Theorems \ref{th:y-exist} and \ref{th:y-exist_2}), from the ones which are strictly needed for
establishing uniqueness (result stated in Theorem \ref{th:yudovich}).

After that, we tackle the proof of our main results.
Section \ref{s:uniform} collects uniform bounds which can be obtained, under our assumptions, on a family of smooth approximate solutions to system \eqref{eq:dd-E}.
In Section \ref{s:exist}, we present the proof of the convergence of smooth approximate solutions to a Yudovich-type solution
to the density-dependent Euler system \eqref{eq:dd-E}, by means of a weak compactness argument.
The proof of the uniqueness result will be carried out in Section \ref{s:uniqueness}.

\subsection*{Notation} 

Let us fix some notation which will be freely used throughout this paper.

Given a Banach space $\mf B$ over $\R^2$, we will adopt the same notation $\mf B(\R^2)$ for scalar, vector-valued and matrix valued functions.
Very often we will simply write $\mf B$, as no confusion can arise in this paper.
Typically, we will resort to the longer notation $\mf B(\R^2)$ when formulating the assumptions and the statements, and in
important centered formulas.

For an interval $I\subset \R$ and $\mf B$ as above,
we denote by $\mc C\big(I;\mf B\big)$ the space of continuous bounded functions on $I$ with values in $\mf B$. For any $p\in[1,+\infty]$,
the symbol $L^p\big(I;\mf B\big)$ stands for the space of measurable functions on $I$ such that the map $t\mapsto \left\|f(t)\right\|_{\mf B}$ belongs to $L^p(I)$.
When $I=[0,T]$, we will often use the shorten notation $L^r_T(\mf B)\,=\,L^r\big([0,T];\mf B\big)$, whereas we will resort to the full notation
in statements and centered formulas.
We will denote the space of test-functions over some set $Q$ equivalently by $\mc C^\infty_0(Q)$ or by $\mc D(Q)$, where $Q$ for us
will typically be either $\R^2$ or $[0,T]\times\R^2$. In the latter case, it is understood that the test function vanishes at time $t=T$.


In our estimates we will often avoid to write the explicit multiplicative constants which allow to pass from one line to the other. Thus, we will write
$A\,\lesssim\, B$ meaning that there exists a universal constant $C>0$, not depending on the solutions nor on the data (in the latter case, this will be pointed out),
such that $A\,\leq\,C\,B$.

Given a sequence of functions $\big(f_n\big)_{n\in\N}$ in a Banach space $\mf B$, we write (as usual) $\big(f_n\big)_{n\in\N}\subset \mf B$ to say that,
for all $n\in\N$, one has $f_n\in\mf B$. If, in addition, there exists a constant $C>0$ such that $\|f_n\|_{\mf B}\leq C$ for all $n\in\N$, we
will write instead $\big(f_n\big)_{n\in\N}\sqsubset \mf B$.


Given a two-dimensional vector field $v$, we define $v^\perp$ to be its rotation of angle $\pi/2$. More precisely, if $v\,=\,\big(v^1,v^2\big)$, then
$v^\perp\,=\,\big(-v^2,v^1\big)$. Similarly, we define the operator $\nabla^\perp$ as $\nabla^\perp\,=\,\big(-\d_2,\d_1\big)$.
Finally, given a two-dimensional vector field $v\,=\,\big(v^1,v^2\big)$, we define its $\curl$ as
$\curl(v)\,=\,\d_1v^2\,-\,\d_2v^1$.


\section*{Acknowledgements}

{\small

This work has been partially supported by the project CRISIS (ANR-20-CE40-0020-01), operated by the French National Research Agency (ANR),
by the Basque Government through the BERC 2022-2025 program and by the Spanish State Research Agency through the BCAM Severo Ochoa excellence accreditation
CEX2021-001142.
Finally, the author also aknowledges the support of the European Union through the COFUND program [HORIZON-MSCA-2022-COFUND-101126600-SmartBRAIN3].

The author is indebted to R. Danchin for sharing insightful comments about a preliminary version of this work.

}

\section{Basic assumptions and main results} \label{s:results}

This section is devoted to the statement of our main results. First of all, let us collect our main assumptions on the initial data.
The needed \tsl{ad hoc} control assumed on the geometric quantity will be directly imposed in the statements.

\subsection{Assumptions on the initial datum} \label{ss:ass-gen}

We present here our main working hypotheses on the initial datum $\big(\rho_0,u_0\big)$. Their motivations, as well as a comparison with
the assumptions of the classical Yudovich theory, have already been discussed in Subsection \ref{ss:nonhom}.

We start by dealing with the density function $\rho_0$. We will avoid the presence of vacuum and, at the same time, require some smoothness on it.

\begin{itemize}
\item[\bf (A1)] The initial density $\rho_0$ verifies, for two suitable positive constants  $0<\rho_*\leq\rho^*$, the property
\[
0\,<\,\rho_*\,\leq\,\rho_0\,\leq\,\rho^*\qquad\qquad \mbox{ and }\qquad\qquad \nabla\rho_0\,\in\,L^\infty(\R^2)\,.
\]
\end{itemize}

Concerning the initial velocity field, we require a finite energy condition, together with (of course) the divergence-free constraint.

\begin{itemize}
 \item[\bf (A2)] The initial velocity field $u_0$ verifies 
\begin{align*}
u_0\,\in\,L^2(\R^2)\,,\qquad\qquad \mbox{ with}\qquad \div u_0\,=\,0\,. 
\end{align*}
\end{itemize}

As mentioned in the introduction, some assumptions on the vorticity of the fluid are needed.
Here we want to distinguish between the conditions implying the existence of a solution (of course, as explained in Subsection \ref{ss:res-overv},
we mean conditional existence, provided the geometric requirements are fulfilled), and those which are needed for the uniqueness
in the considered class.

As far as existence is concerned, we will require the following condition.

\begin{itemize}
\item[\bf (AE3)] Let $u_0$ be the velocity field fixed in assumption \tbf{(A2)}. Define its vorticity $\o_0$ as
$\o_0\,:=\,\curl(u_0)\,=\,\d_1u^2_0\,-\,\d_2u^1_0$. Then $\o_0$ satisfies, for some index $p_0\in\,]2,4]$, the property
\[
\o_0\,\in\,L^{p_0}(\R^2)\,. 
\]
\end{itemize}

In order to guarantee uniqueness, we have to strengthen the previous assumption by requiring also boundedness of $\o_0$. Therefore,
\tbf{(AE3)} will be replaced by the following assumption \tbf{(AU3)}.

\begin{itemize}
\item[\bf (AU3)] Let be $u_0$ the velocity field fixed in assumption \tbf{(A2)}. Define its vorticity $\o_0$ as
$\o_0\,:=\,\curl(u_0)\,=\,\d_1u^2_0\,-\,\d_2u^1_0$. Then $\o_0$ satisfies, for some index $p_0\in\,]2,4]$, the property
\[
\o_0\,\in\,L^{p_0}(\R^2)\,\cap\,L^\infty(\R^2)\,.
\]
\end{itemize}

\begin{rmk} \label{r:u_0-bounded}
Since $p_0>d=2$, under assumption \tbf{(AE3)}, thus also under \tbf{(AU3)}, by Calder\'on-Zygmund theory and Sobolev embeddings one deduces that
\[
 u_0\,\in\,L^\infty(\R^2)\,,\qquad\qquad \mbox{ with }\qquad  \left\|u_0\right\|_{L^\infty}\,\lesssim\,\left\|u_0\right\|_{L^2}\,+\,\left\|\o_0\right\|_{L^{p_0}}\,.
\]
\end{rmk}

Let us point out that the restriction $p_0\leq4$ is purely of technical nature.
This requirement appears only in one precise point of the proof, when dealing with the pressure term: we refer to Proposition \ref{p:press-inf}
for more details.

\subsection{Geometry and approximation} \label{ss:ass-geom}

As explained in the introduction, our study is based on a key notion of \emph{geometric regularity}, already introduced and exploited in \cite{F_2025}. 
To begin with, let us recall its definition.

Assume that a (say) smooth vector field $X$ on $\R^d$ is given, and let $f:\R^d\longrightarrow \R$ be a Lipschitz continuous function.
We are interested in the \emph{directional derivative of $f$ along $X$}, namely in the quantity $\d_Xf$ defined as
\begin{equation*}
\d_Xf\,:=\,X\cdot\nabla f\,=\,\sum_{j=1}^dX^j\,\d_jf\,.
\end{equation*}
Notice that this quantity is well-defined pointwise almost everywhere on $\R^d$ if, for instance,
 $X\in L^\infty$ and $f$ is such that $\nabla f$ belongs to $L^p$, for some $p\in[1,+\infty]$.
This will be enough for our scopes; however notice that, in presence of some conditions on the divergence of $X$
(in our case we will have $\div X=0$, in fact), one could consider even less regular
functions $f$, by defining the quantity $\d_Xf$ in the following weak form:
\[
 \d_Xf\,:=\,\div\big(f\,X\big)\,-\,f\,\div X\,.
\]
Of course, the above definitions applies componentwise in the case when $f$ is replaced
by a vector field $F:\R^d\longrightarrow\R^d$.

As discussed in Subsections \ref{ss:nonhom} and \ref{ss:res-overv}, 
for some of our results we need to impose a sort of \emph{uniform} geometric regularity condition
for families of smooth approximate solutions to the Euler system \eqref{eq:dd-E}.
It goes without saying that those families are constructed by a suitable regularisation of the initial data.
In order to ligthen the presentation and the statement of the main results, let us make this precise and introduce the following definition.

\begin{defin} \label{d:reg-datum}
Let the couple $\big(\rho_0,u_0\big)$ satisfy the set of assumptions \tbf{(A1)} and \tbf{(A2)}. 

A \emph{regularisation of $\big(\rho_0,u_0\big)$} is a family of smooth functions  $\big(\rho_{0,n},u_{0,n}\big)_{n\in\N}$ such that
the following properties are satisfied:
\begin{enumerate}[(i)]
 \item for any $n\in\N$ fixed, $\rho_{0,n}$ belongs to $W^{k,\infty}(\R^2)$ for all $k\geq1$, and verifies the properties
$\rho_*-\veps_n\,\leq\,\rho_{0,n}\,\leq\,\rho^*+\veps_n$, for a suitable positive sequence $\big(\veps_n\big)_{n\in\N}$ such that $\veps_n\longrightarrow0$
when $n\to+\infty$;
 \item for any $n\in\N$ fixed, $u_{0,n}$ belongs to $H^s(\R^2)$ for all $s\geq0$ and verifies $\div u_{0,n}=0$;
 \item one has\footnote{Recall that the symbol $\sqsubset$ means that the sequence is not only included in the respective space, but also bounded therein.}
$\big(\rho_{0,n}\big)_{n\in\N}\,\sqsubset\,W^{1,\infty}(\R^2)$ and $\big(u_{0,n}\big)_{n\in\N}\,\sqsubset\,L^2(\R^2)$,
and there exists a constant $C>0$ such that 
\[
 \sup_{n\in\N}\left\|\rho_{0,n}\right\|_{W^{1,\infty}}\,\leq\,C\,\left\|\rho_{0}\right\|_{W^{1,\infty}}\qquad\qquad \mbox{ and }\qquad\qquad
\sup_{n\in\N}\left\|u_{0,n}\right\|_{L^2}\,\leq\,C\,\left\|u_{0}\right\|_{L^2}\,; 
\]
\item one has (without loss of generality, for the whole sequence) the strong convergence properties
\begin{align*}
&\left\|u_{0,n}\,-\,u_0\right\|_{L^2(\R^2)}\,\longrightarrow\,0\qquad\qquad \mbox{ and }\qquad\qquad
\left\|\rho_{0,n}\,-\,\rho_0\right\|_{L^\infty(K)}\,\longrightarrow\,0
\end{align*}
for any compact set $K\subset\R^2$.
\end{enumerate}

In the case that also assumption \tbf{(AE3)} is satified, one in addition requires the following properties:
\begin{enumerate}
 \item[(v)] the family $\big(u_{0,n}\big)_{n\in\N}$ also verifies $\big(u_{0,n}\big)_{n\in\N}\,\sqsubset\,L^\infty(\R^2)$ 
and, after defining, for all $n\in\N$, $\o_{0,n}\,:=\,\curl(u_{0,n})$, then one has $\big(\o_{0,n}\big)_{n\in\N}\,\sqsubset\,L^{p_0}(\R^2)$,
together with the bound
\[
 \sup_{n\in\N}\left\|\o_{0,n}\right\|_{L^{p_0}}\,\leq\,C\,\left\|\o_{0}\right\|_{L^{p_0}}\,;
\]
\item[(vi)] one has the strong convergence property
\[
\left\|\o_{0,n}\,-\,\o_0\right\|_{L^{p_0}}\,\longrightarrow\,0\,.
\]
\end{enumerate}
Should \tbf{(AE3)} be reinforced as in \tbf{(AU3)}, then one has in addition $\big(\o_{0,n}\big)_{n\in\N}\,\sqsubset\,L^{\infty}(\R^2)$,
with $\sup_{n\in\N}\left\|\o_{0,n}\right\|_{L^\infty}\,\leq\,C\,\left\|\o_{0}\right\|_{L^\infty}$,
and the strong convergence property $\o_{0,n}\,\longrightarrow\,\o_0$ holds true in $L^p$ for any $p\in[p_0,+\infty[\,$.
\end{defin}

Given an initial datum $\big(\rho_0,u_0\big)$ as in the previous definition, the classical regularisation procedures are either a localisation in frequencies \tsl{via}
frequency cut-off operators of a Littlewood-Paley decomposition, or a localisation in space \tsl{via} convolution
with a standard mollifying kernel. In both cases, the properties listed above hold true.
However, we want to keep a general point of view here, as the precise regularisation plays no role in our study and only the regularity and convergence properties
will be really useful. 

\medbreak
Now, let us fix a regularisation of our initial datum $\big(\rho_0,u_0\big)$, as given in Definition \ref{d:reg-datum} above.
Thanks to the results of \cite{D_2010, D:F} (see also Proposition 4.1 in \cite{Brav-F}), for any $n\in\N$ fixed we can solve system \eqref{eq:dd-E} locally in time.
This provides us with a family of smooth solutions $\big(\rho_n,u_n,\nabla\Pi_n\big)_{n\in\N}$, each one defined over $[0,T_n[\,\times\R^2$,
for some positive time $T_n>0$.
Notice that, since the initial datum $\big(\rho_0,u_0\big)$ is not smooth, we cannot guarantee that $\inf_{n\in\N} T_n>0$ without further conditions.

Next, take a family of regular solutions as above. Define $X_n\,:=\,\nabla^\perp\rho_n$.
Observe that both $X_n$ and $u_n$ are smooth, in particular they are defined pointwise over $[0,T_n[\,\times\R^2$.
So, it makes sense to consider the quantity $\d_{X_n}u_n$ and to take its $L^\infty$ norm over $\R^2$.
Let us recall that from the results of \cite{F_2025} it follows that, if $\d_{X_n}u_n$ belongs to $L^1_T(L^\infty)$ for some time $T>0$,
then the solution $\big(\rho_n,u_n,\nabla\Pi_n\big)$ can be continued beyond $T$. 
To some extent, the geometric assumption formulated in \eqref{hyp:geom_unif} below is consistent with this picture.

\subsection{Statement of the main results} \label{ss:results}
We can now state the main results of this paper. 
As already stressed several times, an \tsl{a priori} bound on a specific geometric quantity will be crucial.

Our first result concerns stability of smooth approximate solutions and their convergence to Yudovich-type solutions of the density-dependent
Euler system \eqref{eq:dd-E}. 
By Yudovich-type solutions, we basically mean solutions which are at the level of regularity and integrability specified in the course
of the discussion of Subsection \ref{ss:nonhom}; roughly speaking, this means non-Lipschitz velocity fields $u$ having vorticity in some $L^p$ space
and density functions $\rho\in W^{1,\infty}$.
The precise result is contained in the next statement. We postpone several comments on it after its formulation.

\begin{thm} \label{th:y-exist}
Fix an initial datum $\big(\rho_0,u_0\big)$ such that assumptions {\rm \tbf{(A1)}}, {\rm \tbf{(A2)}} and {\rm \tbf{(AE3)}} are satisfied.
Let $\big(\rho_n,u_n,\nabla\Pi_n\big)_{n\in\N}$ be a sequence of smooth solutions to system \eqref{eq:dd-E}, related to a regularisation
of $\big(\rho_0,u_0\big)$ in the sense of Definition \ref{d:reg-datum} and defined on $[0,T_n[\,\times\R^2$, for some $T_n>0$.
Define the vector field $X_n\,:=\,\nabla\rho_n$.
Assume that there exists a time $T>0$ such that
\begin{equation} \label{hyp:geom_unif}
\inf_{n\in\N}T_n\geq T>0\qquad \mbox{ and }\qquad
\sup_{n\in\N}\int^T_0\left\|\d_{X_n(t)}u_n(t)\right\|_{L^\infty}\,\dt\,<\,+\infty\,.
\end{equation}

Then, up to the extraction of a suitable subsequence, $\big(\rho_n,u_n,\nabla\Pi_n\big)_{n\in\N}$ converges to a triplet $\big(\rho,u,\nabla\Pi\big)$,
defined on $[0,T]\times\R^2$ and such that:
\begin{enumerate}[\rm (i)]
 \item $0<\rho_*\leq\rho(t,x)\leq\rho^*$ and $\nabla \rho\in L^\infty\big([0,T]\times\R^2\big)$;
 \item $u\in L^\infty\big([0,T];L^2(\R^2)\cap L^\infty(\R^2)\big)$; 
 \item the vorticity $\omega\,:=\,\curl u = \d_1u^2-\d_2u^1$ satisfies $\o\in L^\infty\big([0,T];L^{p_0}(\R^2)\big)$; 
 \item $\nabla\Pi\in L^\infty\big([0,T];L^2(\R^2)\big)$. 
\end{enumerate}
In addition, the triplet $\big(\rho,u,\nabla\Pi\big)$ solves \eqref{eq:dd-E}, while
$\big(\rho,\o,\nabla\Pi\big)$ solves the vorticity formulation \eqref{eq:vort} in the weak sense. 
\end{thm}

We can now formulate several remarks concerning Theorem \ref{th:y-exist}.

\begin{rmk} \label{r:result}
Theorem \ref{th:y-exist} can be interpreted as a \emph{conditional} existence result of Yudovich-type solutions to the
incompressible Euler equations with variable density. Indeed, it claims that, provided the \tsl{a priori} geometric assumption \eqref{hyp:geom_unif}
is satisfied, one can construct Yudovich-type solutions to system \eqref{eq:dd-E}.

We remark that the convergence (hence, the conditional existence of Yudovich-type solutions) holds up to the time $T>0$ for which assumption \eqref{hyp:geom_unif}
is verified. Should this assumption hold true with $T=+\infty$, our theorem would yield a global in time result.
\end{rmk}

\begin{rmk} \label{r:convergence}
 The precise topology in which the convergence of $\big(\rho_n,u_n,\nabla\Pi_n\big)_{n\in\N}$ (actually, of a subsequence of it)
towards $\big(\rho,u,\nabla\Pi\big)$ holds is not specified in Theorem \ref{th:y-exist}.
We have chosen to do this, in order to keep the statement lighter and mainly focus on the properties verified by the target triplet.

Roughly speaking, convergence holds true in the weak-$*$ topology of the spaces appearing in items (i) to (iv) of the statement. However, this will be made
precise in the course of the proof, see Sections \ref{s:uniform} and \ref{s:exist}.
\end{rmk}

\begin{rmk} \label{r:furt_prop}
Further properties of the target solutions $\big(\rho,u,\nabla\Pi\big)$ will be derived in the course of the proof.
In particular, we will see that equations \eqref{eq:dd-E} are satisfied almost everywhere on $[0,T]\times\R^2$
(whereas the vorticity formulation is verified only in the weak sense).
\end{rmk}

\begin{rmk} \label{r:geometry}
The geometric condition \eqref{hyp:geom_unif} does not imply the family of velocity fields $\big(u_n\big)_{n\in\N}$ to be uniformly Lipschitz
continuous (which would immediately imply that the constructed solutions are too regular to be compared with Yudovich's ones).
In this sense, it looks as a suitable condition in order to construct a Yudovich theory for equations \eqref{eq:dd-E}.

However, it must be noted that the explicit construction of families of smooth solutions  $\big(\rho_n,u_n,\nabla\Pi_n\big)_{n\in\N}$
satisfying \eqref{hyp:geom_unif}, as well as the characterisation of initial data (below the Lipschitz framework) giving rise to such families of solutions,
has revealed to be, so far, elusive.
\end{rmk}

Let us now focus on the problem of uniqueness of Yudovich-type solutions to equations \eqref{eq:dd-E}.
For this, we need two main ingredients. The first one, already crucial in Yudovich theory, is the property $\o\in L^\infty\big([0,T]\times\R^2\big)$.
The second one, which is proper of the non-homogeneous setting, is to establish that also the limit directional derivative $\d_Xu$,
where $X\,:=\,\nabla^\perp\rho$, remains $L^1_T(L^\infty)$. Unfortunately, under assumption \eqref{hyp:geom_unif},
this property does not follow from weak compactness arguments, therefore this condition must be required directly on the triplet under consideration.

First of all, let us present the following improved version of Theorem \ref{th:y-exist}, which is obtained when replacing
assumption \tbf{(AE3)} with \tbf{(AU3)}. We recall that this simply means that we additionally require the initial vorticity $\o_0$
to be bounded.
The Zygmund space $\mc Z$ and log-Lipschitz space $LL$ appearing in the statement are somehow classical; their definitions will be recalled
in Subsection \ref{ss:further} below.

\begin{thm} \label{th:y-exist_2}
Fix an initial datum $\big(\rho_0,u_0\big)$ such that assumptions {\rm \tbf{(A1)}}, {\rm \tbf{(A2)}} and {\rm \tbf{(AU3)}} are satisfied.
Let $\big(\rho_n,u_n,\nabla\Pi_n\big)_{n\in\N}$ be a sequence of smooth solutions to system \eqref{eq:dd-E}, related to a regularisation
of $\big(\rho_0,u_0\big)$ in the sense of Definition \ref{d:reg-datum} and defined on $[0,T_n[\,\times\R^2$, for some $T_n>0$.
Define the vector field $X_n\,:=\,\nabla\rho_n$.
Assume that there exists a time $T>0$ such that the geometric assumption \eqref{hyp:geom_unif} is satisfied.

Then, up to the extraction of a suitable subsequence, $\big(\rho_n,u_n,\nabla\Pi_n\big)_{n\in\N}$ converges to a triplet $\big(\rho,u,\nabla\Pi\big)$,
defined on $[0,T]\times\R^2$ and such that:
\begin{enumerate}[\rm (a)]
 \item $0<\rho_*\leq\rho(t,x)\leq\rho^*$ and $\nabla \rho\in L^\infty\big([0,T]\times\R^2\big)$;
 \item $u\in L^\infty\big([0,T];L^2(\R^2)\cap \mc Z(\R^2)\big)$; 
in particular $u\in L^\infty\big([0,T];LL(\R^2)\big)$ and, as such, possesses a unique flow; 
 \item the vorticity $\omega$ satisfies $\o\in L^\infty\big([0,T];L^{p_0}(\R^2)\cap L^\infty(\R^2)\big)$;
 \item $\nabla\Pi\in L^\infty\big([0,T];L^2(\R^2)\cap L^\infty(\R^2)\big)$.
\end{enumerate}
In addition, the triplet $\big(\rho,u,\nabla\Pi\big)$ solves \eqref{eq:dd-E}, while
$\big(\rho,\o,\nabla\Pi\big)$ solves the vorticity formulation \eqref{eq:vort} in the weak sense. 
\end{thm}

Finally, we can present the result stating uniqueness of the Yudovich-type solutions to system \eqref{eq:dd-E} ``constructed'' (under assumption
\eqref{hyp:geom_unif}, of course) above.
\begin{thm} \label{th:yudovich}
Fix initial datum $\big(\rho_0,u_0\big)$ such that assumptions {\rm \tbf{(A1)}}, {\rm \tbf{(A2)}} and {\rm \tbf{(AU3)}} are satisfied.
Let the time $T>0$ be given.

Then, there exists at most one Yudovich-type solution $\big(\rho,u,\nabla\Pi\big)$ to system \eqref{eq:dd-E} on $[0,T]\times\R^2$,
namely a solution which satisfies the properties from {\rm (a)} to {\rm (d)} listed in Theorem \ref{th:y-exist_2}, which in addition verifies
\begin{equation} \label{hyp:geom_lim}
 \int^T_0\left\|\d_{X(t)}u(t)\right\|_{L^\infty}\,\dt\,<\,+\infty\,,
\end{equation}
where we have defined $X\,:=\,\nabla^\perp\rho$.
\end{thm}

Before concluding this part, let us formulate an additional comment.

\begin{rmk} \label{r:geom-limit}
The geometric property \eqref{hyp:geom_lim} is an extra condition to be imposed on the target triplet $\big(\rho,u,\nabla\Pi\big)$. Indeed, as already observed,
this property does not follow from \eqref{hyp:geom_unif} by weak compactness arguments: we refer to Remark \ref{r:geom_unif_to-lim} for more comments
in this respect.

However, it is easy to see that, in the case in which \eqref{hyp:geom_unif} is replaced by the stronger condition
\begin{equation} \label{cond:strong_geom}
 \sup_{n\in\N}\left\|\d_{X_n}u_n\right\|_{L^{p_*}_T(L^\infty)}\,<\,+\,\infty\,,
\end{equation}
for some $p_*\in\,]1,+\infty]$, then one can guarantee that the constructed triplet  $\big(\rho,u,\nabla\Pi\big)$ also verifies
the property $\d_Xu\,\in\,L^1\big([0,T];L^\infty(\R^2)\big)$. We refer to Subsection \ref{ss:geom-limit} for more details about this.
\end{rmk}

%

%
%
%

The rest of this work is devoted to the proof of the previous statements.

\section{Uniform bounds} \label{s:uniform}

The assumptions of Theorems \ref{th:y-exist} and \ref{th:y-exist_2} already provide us with a sequence of regularised initial data
$\big(\rho_{0,n},u_{0,n}\big)_{n\in\N}$, satisfying the properties listed in Definition \ref{d:reg-datum}, and with the related sequence of
smooth solutions $\big(\rho_n,u_n,\nabla\Pi_n\big)_{n\in\N}$ to system \eqref{eq:dd-E}, defined on some time interval $[0,T_n[\,$.
Thus, those solutions satisfy, for any $n\in\N$ fixed, the system of equations
\begin{equation} \label{eq:dd-E_n}
\left\{\begin{array}{l}
\d_t\rho_n\,+\,u_n\cdot\nabla\rho_n\,=\,0 \\[1ex]
\rho_n\,\d_tu_n\,+\,\rho_n\,(u_n\cdot\nabla) u_n\,+\,\nabla\Pi_n\,=\,0 \\[1ex]
\div u_n\,=\,0\,.
       \end{array}
\right.
\end{equation}
In addition, according to \eqref{hyp:geom_unif}, we have that $T_n\geq T>0$ for all $n\in\N$ and,
after setting $X_n\,:=\,\nabla^\perp\rho_n$, with $\nabla^\perp\,=\,\big(-\d_2,\d_1\big)$, that
\begin{equation} \label{assump:weaker}
M\,:=\,\sup_{n\in\N}\;\int^T_0\left\|\d_{X_n}u_n(t)\right\|_{L^\infty}\,\dt\,<\,+\,\infty\,.
\end{equation}

The goal of this section is to exhibit \emph{uniform bounds}, in suitable norms, for $\big(\rho_n,u_n,\nabla\Pi_n\big)_{n\in\N}$.
This is the very first step in order to prove Theorems \ref{th:y-exist} and \ref{th:y-exist_2}.
In the first part of this section we will work with assumptions \tbf{(A1)}, \tbf{(A2)} and \tbf{(AE3)} only.
In the last part, namely Subsection \ref{ss:further-bounds} below, instead, we will replace assumptions \tbf{(AE3)} with \tbf{(AU3)}, and we will derive
further uniform bounds for the sequence of smooth approximate solutions.

Since most of the bounds will depend on the quantity $M$ defined in \eqref{assump:weaker}, from now on we will directly refer to that relation, instead of quoting
the original assumption \eqref{hyp:geom_unif} appearing in the statement of the theorems.

\subsection{The case of unbounded vorticities} \label{ss:unif-unbounded}

As announced above, 
we are now going to work under hypotheses \tbf{(A1)}, \tbf{(A2)} and \tbf{(AE3)}, plus of course \eqref{assump:weaker}.
In particular, in this part we do \emph{not} assume to dispose of a $L^\infty$ information for $\o_0$.
The properties of the approximate initial data listed in Definition \ref{d:reg-datum} will be freely used here.

\subsubsection{Transport and energy bounds}
First of all, thanks to assumption \tbf{(A1)} and the divergence-free condition over $u_n$,
from the transport equation satisfied by $\rho_n$ one immediately obtains that
\begin{equation} \label{est:rho-inf}
\forall\,n\in\N\,,\quad\forall\,(t,x)\in[0,T]\times\R^2\,,\qquad\qquad \rho_*-\veps_n\,\leq\,\rho_n(t,x)\,\leq\,\rho^*+\veps_n\,,
\end{equation}
for a suitable sequence $\big(\veps_n\big)_{n\in\N}$ of positive real numbers such that $\veps_n\longrightarrow0$ for $n\to+\infty$
(recall the first item in Definition \ref{d:reg-datum}).

Next, we perform an energy estimate on the momentum equation, namely the second equation appearing in system \eqref{eq:dd-E_n}.
Since the velocity field $u_n$ is divergence-free, the pressure term identically vanishes when multiplied in $L^2$ by $u_n$. Standard computations thus yield
\[
\frac{\dd}{\dd t}\int_{\R^2}\rho_n\,\left|u_n\right|^2\,\dd x\,=\,0\,,
\]
which implies, after an integration in time, the following equality, for all $t\in[0,T]$:
\[ 
\left\|\sqrt{\rho_n(t)}\,u_n(t)\right\|_{L^2}^2\,=\,\left\|\sqrt{\rho_{0,n}}\,u_{0,n}\right\|_{L^2}^2\,.
\] 
Using \eqref{est:rho-inf} and the properties on the initial datum, we immediately deduce that
\begin{equation} \label{est:u-2}
\big(u_n\big)_{n\in\N}\,\sqsubset\,L^\infty\big([0,T];L^2(\R^2)\big)\,,
\end{equation}
where we recall that the symbol $\sqsubset$ means that the sequence is not only included, but also \emph{bounded} in the respective space.

\subsubsection{The gradient of the densities}

After setting $X_n\,:=\,\nabla^\perp\rho_n$ as usual in this work, straightforward computations show that
the vector field $X_n$ is transported (in the sense of $1$-forms) by the flow of $u_n$, namely it solves
\begin{equation} \label{eq:X}
 \d_tX_n\,+\,(u_n\cdot\nabla) X_n\,=\,\d_{X_n}u_n\,.
\end{equation}

From classical transport estimates in Lebesgue spaces applied to equation \eqref{eq:X}, we deduce a bound for $\|X_n\|_{L^\infty}$,
whence for $\|\nabla\rho_n\|_{L^\infty}$: for any $t\in[0,T]$, one has
\[ 
\left\|\nabla\rho_n(t)\right\|_{L^\infty}\,\leq\,\left\|\nabla\rho_{0,n}\right\|_{L^\infty}\,+\,\int^t_0\left\|\d_{X_n}u_n(\t)\right\|_{L^\infty}\,\dd\tau\,.
\] 
Therefore, in view of assumption \eqref{assump:weaker}, we infer that
\begin{equation} \label{est:unif_D-rho}
\big(\nabla \rho_n\big)_{n\in\N}\,\sqsubset\,L^\infty\big([0,T]\times\R^2\big)\,.
\end{equation}

\subsubsection{Use of the momentum and its ``vorticity''}

Following \cite{F_2025}, 
we now consider the momentum $m_n\,:=\,\rho_n\,u_n$. Observe that the Helmholtz decomposition of $m_n$ writes as
\begin{equation} \label{eq:Helm-mom}
m_n\,=\,\rho_n\,u_n\,=\,-\,\nabla^\perp(-\Delta)^{-1}\eta_n\,-\,\nabla(-\Delta)^{-1}\big(u_n\cdot\nabla\rho_n\big)\,,
\end{equation}
where we have used the property $\div u_n\,=\,0$, so that $\div m_n\,=\,u_n\cdot\nabla\rho_n$,
and where we have defined the scalar field $\eta_n$ as the ``vorticity'' of the momentum, that is
\[
 \eta_n\,:=\,\curl m_n\,=\,\d_1\big(m_n^2\big)\,-\,\d_2\big(m_n^1\big)\,.
\]

We remark that we can recover the true vorticity of the fluid $\o_n$ from $\eta_n$ \tsl{via} the relation
\begin{equation} \label{eq:eta-vort}
 \eta_n\,=\,\rho_n\,\omega_n\,+\,u_n\cdot\nabla^\perp\rho_n\,.
\end{equation}
Thus, thanks to \eqref{est:rho-inf}, for any $n\in\N$ and any $p\in[1,+\infty]$, we can formally bound
\begin{align} \label{est:o_n-eta_n}
&\left\|\eta_n\right\|_{L^p}\,\lesssim\,\left\|\o_n\right\|_{L^p}\,+\,\left\|u_n\right\|_{L^p}\,\left\|\nabla\rho_n\right\|_{L^\infty} \\[1ex]
\nonumber
&\qquad\qquad \mbox{ and }\qquad\qquad
\left\|\o_n\right\|_{L^p}\,\lesssim\,\left\|\eta_n\right\|_{L^p}\,+\,\left\|u_n\right\|_{L^p}\,\left\|\nabla\rho_n\right\|_{L^\infty}\,,
\end{align}
for suitable implicit multipicative constants depending only on the positive numbers $\rho_*$ and $\rho^*$ appearing in assumption \tbf{(A1)}.

In addition, we see that, for any $n\in\N$, the vorticity $\eta_n$ of the momentum $m_n$ solves the transport equation
\begin{equation} \label{eq:eta}
 \d_t\eta_n\,+\,u_n\cdot\nabla\eta_n\,=\,\d_{X_n}u_n\cdot u_n\,.
\end{equation}
We now use the previous equation to establish suitable $L^p$ bounds for $\eta_n$.
More precisely, fixed $p_0>2$ as in assumption \tbf{(AE3)},
by transport estimates we infer that, for any $q\in[p_0,+\infty]$ and any $t\in[0,T]$, one has
\begin{equation} \label{est:eta-L_first}
\left\|\eta_n(t)\right\|_{L^{q}}\,\leq\,\left\|\eta_{0,n}\right\|_{L^{q}}\,+\,\int^t_0\left\|\d_{X_n} u_n(\t)\right\|_{L^\infty}\,
\left\|u_n(\t)\right\|_{L^{q}}\,\dt\,,
\end{equation}
where we have defined $\eta_{0,n}\,:=\,\curl\big(\rho_{0,n}\,u_{0,n}\big)$.
Owing to \eqref{est:o_n-eta_n}, from assumptions \tbf{(A1)} and \tbf{(A2)} and the condition $p_0>2$ we deduce that there exists a constant $C>0$ such that,
for any $q\in[p_0,+\infty]$, one formally has
\begin{equation} \label{est:eta-o_0}
\left\|\eta_{0,n}\right\|_{L^q}\,\leq\,\rho^*\,\left\|\o_{0,n}\right\|_{L^q}\,+\,
\left\|u_{0,n}\right\|_{L^2\cap L^\infty}\,\left\|\nabla\rho_{0,n}\right\|_{L^\infty}\,\leq\,C\,.
\end{equation}
Since we are working under assumption \tbf{(AE3)}, which ensures $\o_0\in L^{p_0}$ only, from now on we will use the previous estimate \eqref{est:eta-o_0}
and estimate \eqref{est:eta-L_first} with $q=p_0$.
However, their general forms with varying $q\in[p_0,+\infty]$ will be
useful later, when replacing \tbf{(AE3)} with \tbf{(AU3)}, see Subsection \ref{ss:further-bounds} below: there, we will assume also
$\o_0\in L^\infty$, hence we will be able to apply \eqref{est:eta-L_first} also for $q=+\infty$.

\subsubsection{Bounds for the velocities and the vorticities}
In view of \eqref{est:eta-L_first}, we need to find a bound for the velocity field $u_n$ in the space $L^\infty_T(L^{p_0})$.
Owing to \eqref{est:u-2}, it is enough to bound its $L^\infty$ norm (both in time and space).
For this, we will use the momentum $m_n$ and its vorticity $\eta_n$.

As a matter of fact, equality \eqref{eq:Helm-mom}, combined with the condition $p_0>2=d$, yields that,
for any $n\in\N$ and any $t\in[0,T]$ fixed (which we omit from the notation), we can write
\begin{align*}
 \left\|m_n\right\|_{L^\infty}\,&\lesssim\,\left\|m_n\right\|_{W^{1,p_0}}\,=\,\left\|\rho_n\,u_n\right\|_{W^{1,p_0}} \\
 &\lesssim\,\left\|\rho_n\,u_n\right\|_{L^{p_0}}\,+\,\left\|\eta_n\right\|_{L^{p_0}}\,+\,\left\|u_n\cdot\nabla\rho_n\right\|_{L^{p_0}}\,,
\end{align*}
where we have also used the Calder\'on-Zygmund theory to bound the $L^{p_0}$ norm of $\nabla m_n$ by the sum of the $L^{p_0}$ norms of $\eta_n$ and $u_n\cdot\nabla \rho_n$.
At this point, the H\"older inequality and an interpolation argument yield
\begin{align*}
\left\|m_n\right\|_{L^\infty}\,&\lesssim\,\rho^*\,\left\|u_n\right\|_{L^{p_0}}\,+\,
\left\|\eta_n\right\|_{L^{p_0}}\,+\,\left\|u_n\right\|_{L^{p_0}}\,\left\|\nabla\rho_n\right\|_{L^\infty} \\
&\lesssim\,\left\|u_n\right\|^\theta_{L^2}\,\left\|u_n\right\|^{1-\theta}_{L^\infty}\,+\,\left\|\eta_n\right\|_{L^{p_0}}\,,
\end{align*}
where the exponent $\theta$ is defined as $\theta\,=\,2/p_0\,\in\;]0,1[\,$. Notice that, in passing from the first inequality to the second one,
we have made use of the uniform bound established in \eqref{est:unif_D-rho}.

Remark that, owing to \eqref{est:rho-inf}, we have $\left\|u_n\right\|_{L^p}\,\approx\, \left\|m_n\right\|_{L^p}$ for any $p\in[2,+\infty]$.
This fact, together with the previous estimate for $m_n$ in $L^\infty$ and inequality \eqref{est:u-2}, yields, after an application of the Young inequality,
the bound
\begin{equation} \label{est:u-inf_prelim}
\left\|u_n\right\|_{L^\infty}\,\lesssim\,1\,+\,\left\|\eta_n\right\|_{L^{p_0}}\,,
\end{equation}
where the implicit multiplicative constant now depends not only on $\rho_*$ and $\rho^*$, but also on the norms $\|u_0\|_{L^2}$ and $\|\nabla\rho_0\|_{L^\infty}$,
and on the quantity $M$ defined in \eqref{assump:weaker}.
In particular, the previous estimate implies that
\[
 \left\|u_n\right\|_{L^{p_0}}\,\lesssim\,\left\|u_n\right\|_{L^2\cap L^\infty}\,\lesssim\,1\,+\,\left\|\eta_n\right\|_{L^{p_0}}\,.
\]

At this point, we can insert this last inequality into \eqref{est:eta-L_first}, where we take $q=p_0$. 
From an application of the Gr\"onwall lemma and assumption \eqref{assump:weaker}, we deduce that
\begin{equation} \label{est:eta_Leb}
\big(\eta_n\big)_{n\in\N}\,\sqsubset\,L^\infty\big([0,T];L^{p_0}(\R^2)\big)\,.
\end{equation}
Going back to \eqref{est:u-inf_prelim}, in turn we infer that
\begin{equation} \label{est:u-inf}
\big(u_n\big)_{n\in\N}\,\sqsubset\,L^\infty\big([0,T]\times \R^2\big)\,.
\end{equation}

Now, gathering estimates \eqref{est:unif_D-rho}, \eqref{est:o_n-eta_n}, \eqref{est:eta_Leb} and \eqref{est:u-inf} all together, we find
\begin{equation} \label{est:vort_Leb}
 \big(\o_n\big)_{n\in\N}\,\sqsubset\,L^\infty\big([0,T];L^{p_0}(\R^2)\big)\,.
\end{equation}
This uniform bound, together with \eqref{est:u-2} and \eqref{est:u-inf} again, in turn implies that
\begin{equation} \label{est:u_W^1}
 \big(u_n\big)_{n\in\N}\,\sqsubset\,L^\infty\big([0,T];W^{1,p_0}(\R^2)\big)\,.
\end{equation}

\subsubsection{Uniform bounds for the pressure gradients}
Observe that a pressure term appears in the vorticity equation \eqref{eq:vort}, hence
finding Yudovich-type solutions for the density-dependent system \eqref{eq:dd-E} requires to deal with the term $\nabla\Pi$.
Therefore, let us establish some uniform bounds for the sequence of $\nabla\Pi_n$'s.
To begin with, we recall that
\begin{align*}
&\big(u_n\big)_{n\in\N}\,\sqsubset\,L^\infty\left([0,T];L^2(\R^2)\cap L^\infty(\R^2)\right) 
\qquad \mbox{ and }\qquad
\big(\nabla u_n\big)_{n\in\N}\,\sqsubset\,L^\infty\big([0,T];L^{p_0}(\R^2)\big)\,.
\end{align*}
Hence, one has that the sequence of the products $(u_n\cdot\nabla)u_n$ is uniformly bounded in any space
$L^\infty_T(L^q)$, for any $q\in[q_0,p_0]$, where we have defined
\[
\frac{1}{q_0}\,:=\,\frac{1}{2}\,+\,\frac{1}{p_0}\,.
\]
Notice that, as $2<p_0<+\infty$, one has $1<q_0<2$. Therefore, in turn we deduce that
\begin{equation} \label{est:bilin-2}
 \big((u_n\cdot\nabla) u_n\big)_{n\in\N}\,\sqsubset\,L^\infty\big([0,T];L^2(\R^2)\big)\,.
\end{equation}
At this point, we can divide the momentum equation in \eqref{eq:dd-E_n} by $\rho_n$ and apply the divergence operator. We find that
$\nabla\Pi_n$ satisfies the following elliptic equation:
\[
 -\,\div\left(\frac{1}{\rho_n}\,\nabla\Pi_n\right)\,=\,\div\Big(u_n\cdot\nabla u_n\Big)\,.
\]
Thanks to Lemma 2 of \cite{D_2010} (based on an application of the Lax-Milgram theorem) and the uniform bounds established in \eqref{est:bilin-2},
we infer that
\begin{equation} \label{est:unif-press}
  \big(\nabla\Pi_n\big)_{n\in\N}\,\sqsubset\,L^\infty\big([0,T];L^2(\R^2)\big)\,.
\end{equation}

\subsection{Further estimates for bounded initial vorticities} \label{ss:further-bounds}

In this part, we replace assumption \tbf{(AE3)} with assumption \tbf{(AU3)}. With respect to the previous subsection, this simply amounts
to suppose furthermore that $\o_0$ belongs to $L^\infty$. This enables us to derive
additional uniform bounds for $\big(\rho_n,u_n,\nabla\Pi_n\big)_{n\in\N}$.

First of all, we see that, similarly to \eqref{est:eta_Leb}, combining \eqref{est:eta-L_first}, \eqref{est:eta-o_0} and \eqref{est:u-inf} all together yields
\begin{equation} \label{est:eta_inf}
 \big(\eta_n\big)_{n\in\N}\,\sqsubset\,L^\infty\big([0,T]\times\R^2\big)\,.
\end{equation}
Using this information and taking advantage of \eqref{est:o_n-eta_n}, \eqref{est:u-inf} and \eqref{est:unif_D-rho}, we deduce a similar property for the vorticity
fields $\o_n$: we have
\begin{equation} \label{est:o_inf}
 \big(\o_n\big)_{n\in\N}\,\sqsubset\,L^\infty\big([0,T]\times\R^2\big)\,.
\end{equation}

Now, by Calder\'on-Zygmund theory, the bounds in \eqref{est:vort_Leb} and \eqref{est:o_inf} imply that
\[
 \big(\nabla u_n\big)_{n\in\N}\,\sqsubset\,\bigcap_{p_0\leq q<+\infty\,} L^\infty\big([0,T];L^q(\R^2)\big)\,.
\]
Thanks to this and using \eqref{est:unif_D-rho} again, we obtain the following property:
\begin{equation} \label{est:geom-unif}
 \big(\d_{X_n}u_n\big)_{n\in\N}\,\sqsubset\,\bigcap_{p_0\leq q<+\infty\,} L^\infty\big([0,T];L^q(\R^2)\big)\,\cap\,L^1\big([0,T];L^\infty(\R^2)\big)\,,
\end{equation}
where the last property obviously follows from assumption \eqref{assump:weaker}.

\begin{rmk} \label{r:geom_unif_to-lim}
The last information $\big(\d_{X_n}u_n\big)_{n\in\N}\,\sqsubset\,L^1_T(L^\infty)$ appearing in \eqref{est:geom-unif} has not good properties when
passing to the limit for $n\to+\infty$. However, as it appears clear from the computations of Subsection \ref{ss:unif-unbounded}, the property
$\d_Xu\in L^1_T(L^\infty)$ plays a major role for establishing uniform bounds; analogously, it will be a key ingredient also for proving uniqueness.
This justifies the extra condition \eqref{hyp:geom_lim} in Theorem \ref{th:yudovich}.
\end{rmk}

On the other hand, let us also formulate the next observation, related to Remark \ref{r:geom-limit}.

\begin{rmk} \label{r:geom_unif_stronger}
If assumption \eqref{assump:weaker}, that is \eqref{hyp:geom_unif}, was replaced by the stronger
assumption \eqref{cond:strong_geom} which appears in that remark, then property \eqref{est:geom-unif}
would become
\begin{equation} \label{est:geom_unif-str}
 \big(\d_{X_n}u_n\big)_{n\in\N}\,\sqsubset\,\bigcap_{p_0\leq q<+\infty\,} L^\infty\big([0,T];L^q(\R^2)\big)\,\cap\,L^{p_*}\big([0,T];L^\infty(\R^2)\big)\,.
\end{equation}
In this case, we would be able to \emph{deduce} property \eqref{hyp:geom_lim} for the solution we are going to construct. We refer to
Proposition \ref{p:geom-lim_tot} below for more details.
\end{rmk}

\section{Proof of Theorem \ref{th:y-exist}: convergence} \label{s:exist}

In this section, we pass to the limit $n\to+\infty$ in the weak formulation of equations \eqref{eq:dd-E_n} and prove Theorem \ref{th:y-exist}.
First of all, in Subsection \ref{ss:limit} we identify the limit profiles $\big(\rho,u,\nabla\Pi\big)$. Then, in Subsection \ref{ss:limit-eq}
we show convergence in the mass and momentum equations, thus proving that the triplet $\big(\rho,u,\nabla\Pi\big)$ is indeed a solution of \eqref{eq:dd-E}
verifying the properties listed in Theorem \ref{th:y-exist}. In Subsection \ref{ss:vort-eq}, instead, we discuss the validity of the vorticity equation
\eqref{eq:vort}.
Notice that, in all these arguments, we will use only the uniform boundedness properties established
in Subsection \ref{ss:unif-unbounded}; in particular, we do not assume $\o_0\in L^\infty$ in this part.

Finally, Subsection \ref{ss:geom-limit} discusses the properties which can be inferred on the geometric quantity $\d_Xu$ in the limit. As strictly related
to this discussion, at some point we will show upgrades of those properties which can be obtained assuming also the property $\o_0\in L^\infty$,
and we will justify the claim of Remark \ref{r:geom-limit}.

\subsection{Identification of the limit profiles} \label{ss:limit}

As a preliminary step, we need to establish weak and strong convergence properties for the family of smooth approximate solutions
$\big(\rho_n,u_n,\nabla\Pi_n\big)_{n\in\N}$, towards some limit profile $\big(\rho,u,\nabla\Pi\big)$.
This is the goal of the present subsection.

\subsubsection{Convergence of the density functions}
Let us focus on the density functions first. From the uniform properties \eqref{est:rho-inf} and \eqref{est:unif_D-rho} and an inspection
of the mass equation, also based on \eqref{est:u-inf}, we deduce that there exists
a function
\begin{align*}
& \rho\,\in\,W^{1,\infty}\big([0,T]\times\R^2\big)\,, 
\qquad\qquad \mbox{ with } \qquad 0<\rho_*\leq\rho\leq\rho^*\,,
\end{align*}
such that, up to the extraction of a suitable subsequence, one has
\begin{equation} \label{conv:dens-weak}
 \rho_n\,\stackrel{*}{\wtend}\,\rho \qquad \mbox{ and }\qquad \nabla\rho_n\,\stackrel{*}{\wtend}\,\nabla\rho \qquad\qquad \mbox{ in }
\qquad L^\infty\big([0,T]\times\R^2\big)\,.
\end{equation}

Moreover, from the equation
\[ 
 \d_t\rho_n\,=\,-\,\div\big(\rho_n\,u_n\big)
\] 
and properties \eqref{est:rho-inf} and \eqref{est:u-2}, we deduce that $\big(\d_t\rho_n\big)_{n\in\N}$ is bounded in $L^\infty\big([0,T];H^{-1}(\R^2)\big)$.
As, for any compact set $K\subset \R^2$, one has $W^{1,\infty}(K)\subset H^1(K)$ and the compact embedding $H^1(K)\hookrightarrow\hookrightarrow L^2(K)$,
from Ascoli-Arzel\`a theorem one deduces the strong convergence
\begin{equation} \label{conv:dens-strong}
 \rho_n\,\longrightarrow\,\rho\qquad\qquad \mbox{ in }\qquad \mc C\big([0,T];L^2_{\rm loc}(\R^2)\big)\,.
\end{equation}
Notice that, owing to the uniform bound of the $\rho_n$'s in $L^\infty\big([0,T]\times\R^2\big)$, by interpolation one also infers the strong convergence
\begin{equation} \label{conv:dens-more}
\forall\,p\in[2,+\infty[\;,\qquad\qquad\qquad \rho_n\,\longrightarrow\,\rho\qquad\qquad \mbox{ in }\qquad \mc C\big([0,T];L^p_{\rm loc}(\R^2)\big)\,.
\end{equation}

For later use, we conclude this part by mentioning another strong convergence result concerning the density functions.
For any $n\in\N$, define
\[
 a_n\,:=\,\frac{1}{\rho_n}\,.
\]
Using that $\nabla a_n\,=\,-\,\rho_n^{-2}\,\nabla\rho_n$ and
repeating \tsl{mutatis mutandis} the same arguments as above, one easily deduces also that
\begin{equation} \label{conv:inv-dens-str}
\frac{1}{\rho_n}\,=\,a_n\,\longrightarrow\,a\,=\,\frac{1}{\rho}\qquad\qquad \mbox{ in }\qquad \mc C\big([0,T];L^2_{\rm loc}(\R^2)\big)\,.
\end{equation}

\subsubsection{Weak convergence of the velocities and pressure gradients}

Next, focus on the velocity fields $u_n$. From \eqref{est:u-2} we know that there exists a target profile $u\in L^\infty\big([0,T];L^2(\R^2)\big)$
such that, up to the extraction of a suitable subsequence, one has
\begin{equation} \label{conv:u-2}
 u_n\,\stackrel{*}{\wtend}\,u\qquad\qquad \mbox{ in }\qquad L^\infty\big([0,T];L^2(\R^2)\big)\,.
\end{equation}
At the same time, by \eqref{est:vort_Leb}, one also knows that
\begin{equation} \label{conv:o}
 \o_n\,\stackrel{*}{\wtend}\,\o\qquad\qquad \mbox{ in }\qquad L^\infty\big([0,T];L^{p_0}(\R^2)\big)\,,
\end{equation}
for a suitable scalar function $\o$ belonging to that space.
Of course, owing to the linearity of the Biot-Savart law, we deduce that
\[
 \o\,=\,\curl u\,.
\]

\begin{rmk} \label{r:o-limit}
Should one have assumed $\o_0\in L^\infty$, then, in addition, one would have got
$\o\in L^\infty\big([0,T]\times\R^2\big)$, owing to \eqref{est:o_inf} and the uniqueness of the weak limit.
\end{rmk}

Finally, by \eqref{est:u-inf} and \eqref{est:u_W^1} and the uniqueness of the weak limit, we infer that
$u$ belongs also to $L^\infty\big([0,T]\times\R^2\big)$ and $L^\infty\big([0,T];W^{1,p_0}(\R^2)\big)$
and that the weak-$*$ convergence holds true also in those spaces:
\begin{equation} \label{conv:u-more}
 u_n\,\stackrel{*}{\wtend}\,u\qquad\qquad \mbox{ in }\qquad L^\infty\big([0,T]\times\R^2\big)\qquad \mbox{ and }\qquad L^\infty\big([0,T];W^{1,p_0}(\R^2)\big)\,.
\end{equation}

Before going on, let us focus on \eqref{est:unif-press}: from it, we deduce the existence of a gradient $\nabla\Pi$
in the space $L^\infty\big([0,T];L^2(\R^2)\big)$ such that (omitting another extraction) one has
\begin{equation} \label{conv:press-L^2}
 \nabla\Pi_n\,\stackrel{*}{\wtend}\,\nabla\Pi\qquad\qquad \mbox{ in }\qquad L^\infty\big([0,T];L^2(\R^2)\big)\,.
\end{equation}
The (important) consequence of this latter property is twofold. On the one had, it allows us to pass to the limit in the weak formulation of
equations \eqref{eq:dd-E_n} without requiring the divergence-free condition on the test-function used for the momentum equation.
On the other hand, it permits to pass to the limit also in the vorticity equation related to \eqref{eq:dd-E_n}, which we will do in Subsection \ref{ss:vort-eq} below.

\subsubsection{Strong convergence of the velocity fields}

Let us go back to \eqref{est:unif-press} for a while. 
Another important consequence of it is that it allows to establish strong convergence properties for the sequence of velocity fields.
To this scope, we rewrite the momentum equation in \eqref{eq:dd-E_n} as
\begin{equation} \label{eq:d_tu_n}
 \d_tu_n\,=\,-\,(u_n\cdot\nabla) u_n\,-\,\frac{1}{\rho_n}\,\nabla\Pi_n\,.
\end{equation}
Owing to \eqref{est:rho-inf}, \eqref{est:bilin-2} and \eqref{est:unif-press}, from it we deduce that $\big(\d_tu_n\big)_{n\in\N}$ is a uniformly bounded sequence in
the space $L^\infty\big([0,T];L^2(\R^2)\big)$. This implies that
\[
 \big(u_n\big)_n\,\sqsubset\,W^{1,\infty}\big([0,T];L^2(\R^2)\big)\,\cap\,L^\infty\big([0,T];W^{1,p_0}(\R^2)\big)\,,
\]
where we have also used \eqref{est:u_W^1}. Since $p_0>2$, we have that $W^{1,p_0}_{\rm loc}\hookrightarrow H^1_{\loc}$. Hence, the compact
embedding $H^1_\loc\hookrightarrow\hookrightarrow L^2_\loc$ and the Ascoli-Arzel\`a theorem imply
(up to an omitted extraction) the strong convergence property
\begin{equation} \label{conv:strong-u}
u_n\,\longrightarrow\,u\qquad\qquad \mbox{ in }\qquad \mc C\big([0,T];L^2_\loc(\R^2)\big)\,.
\end{equation}

\subsection{Passing to the limit in the equations} \label{ss:limit-eq}

In this subsection, we use the previous convergence properties to pass to the limit in the weak formulation of equations \eqref{eq:dd-E_n}, and
in turn prove that the identified target profile $\big(\rho,u,\nabla\Pi\big)$ is a Yudovich-type solution of system \eqref{eq:dd-E}.
In order to get closer to Yudovich theory, in Subsection \ref{ss:vort-eq} we discuss the validity of the vorticity formulation of
the momentum equation, see \eqref{eq:vort}.

\subsubsection{The limit of the mass equation} \label{sss:dens-eq}

We start by dealing with the mass conservation equation, namely the first equation appearing in \eqref{eq:dd-E_n}.
Let us fix a scalar test-function $\phi\in \mc D\big(\R_+\times\R^2\big)$ such that $\supp\phi\subset[0,T]\times K$, for
a comptact set $K\subset\R^2$. Notice that, in particular, one has $\phi(T)\equiv0$.
We want to pass to the limit in the relation
\begin{align}
\label{eq:weak-form-dens}
&-\int^T_0\int_{\R^2}\rho_n\,\d_t\phi\,\dx\,\dt\,-\,\int^T_0\int_{\R^2}\rho_n\,u_n\cdot\nabla\phi\,\dx\,\dt\,=\,
\int_{\R^2}\rho_{0,n}\,\phi(0,\cdot)\,\dx\,.
\end{align}

The convergence of the initial datum is straightforward, thanks to the properties listed in Definition \ref{d:reg-datum} and implicitly assumed
in Theorem \ref{th:y-exist}.
Also the convergence of the the first term on the left-hand side is easy, because this term is linear: for it,
one can use \eqref{conv:dens-weak}, for instance.
For the non-linear term $\rho_n\,u_n$ appearing in \eqref{eq:weak-form-dens},
instead, we can use the weak-$*$ convergences \eqref{conv:u-2} and \eqref{conv:u-more} for the velocities,
together with the strong convergence \eqref{conv:dens-more} for the densities, to get that
\begin{equation} \label{conv:momentum}
 \rho_n\,u_n\,\stackrel{*}{\wtend}\,\rho\,u\qquad\qquad \mbox{ in }\qquad L^\infty\big([0,T];L^2(K)\big)\,,
\end{equation}
where we recall that $K$ is the (compact) spatial support of the test-function $\phi$ fixed above.

In the end, we have proven that the target profiles $\rho$ and $u$ satisfy the equation
\begin{align*}
&-\int^T_0\int_{\R^2}\rho\,\d_t\phi\,\dx\,\dt\,-\,\int^T_0\int_{\R^2}\rho\,u\cdot\nabla\phi\,\dx\,\dt\,=\,
\int_{\R^2}\rho_{0}\,\phi(0,\cdot)\,\dx
\end{align*}
for any test-function $\phi\in \mc D\big(\R_+\times\R^2\big)$ with compact support in $[0,T]\,\times\R^2$. In other words, they solve the mass equation in \eqref{eq:dd-E}
in the weak form.

\subsubsection{The limit of the momentum equation} \label{sss:vel-eq}

Next, we pass to the limit in the weak formulation of the momentum equation in \eqref{eq:dd-E_n}. For this, differently from what is classically done
in incompressible fluid dynamics,
we use test-functions which are not necessarily divergence-free. This is possible thanks to the regularity properties established for the pressure gradients and,
at the same time, it is important in order to recover also the vorticity formulation of the momentum equation.

So, let us take a smooth vector field $\psi\in \mc D\big(\R_+\times\R^2\big)$, such that $\supp\psi\subset[0,T]\times K$, for some compact set $K\subset\R^2$.
Observe that, as before, one has $\psi(T)\equiv0$. Testing the second equation in \eqref{eq:dd-E_n} against such a $\psi$ yields
\begin{align}
\label{eq:weak-mom}
&-\int^T_0\int_{\R^2}\rho_n\,u_n\cdot\d_t\psi\,\dx\,\dt\,-\,\int^T_0\int_{\R^2}\rho_n\,u_n\otimes u_n:\nabla\psi\,\dx\,\dt \\
\nonumber
&\qquad\qquad\qquad\qquad\qquad +\int^T_0\int_{\R^2}\nabla\Pi_n\cdot\psi\,\dx\,\dt\,=\,\int_{\R^2}\rho_{0,n}\,u_{0,n}\cdot\psi(0,\cdot)\,\dx\,.
\end{align}

Once again, the convergence of the term depending on the initial datum is easy, owing to the approximation properties listed in Definition \ref{d:reg-datum}.
The convergence of the pressure term, instead, simply relies on \eqref{conv:press-L^2}, as this term is linear. In addition, the term presenting the
time derivative $\d_t\psi$
can be dealt with by invoking the convergence property established in \eqref{conv:momentum}.

So, we only have to deal with the limit for $n\to+\infty$ of the convective term $\rho_n u_n\otimes u_n$. We claim that it converges
to the expected limit $\rho u\otimes u$. For showing this, we decompose the difference of the two terms as follows:
\begin{align*}
\rho_n\,u_n\otimes u_n\,-\,\rho\,u\otimes u\,&=\,\Big(\rho_n\,-\,\rho\Big)\,u_n\otimes u_n\,+\,\rho\,\Big(u_n\,-\,u\Big)\otimes u_n\,+\,
\rho\,u\otimes\Big(u_n\,-\,u\Big) \\
&=\,I_1\,+\,I_2\,+\,I_3\,.
\end{align*}
We now treat each term separately. First of all, thanks to the weak convergence property \eqref{conv:u-2} and the fact that
both $\rho$ and $u$ are bounded on $[0,T]\times\R^2$ (keep in mind \eqref{conv:dens-weak} and \eqref{conv:u-more} above), we see that
\[
 I_3\,\stackrel{*}{\wtend}\,0\qquad\qquad \mbox{ in }\qquad L^\infty\big([0,T];L^2(\R^2)\big)\,.
\]
Next, we deal with $I_1$. Recall that $K$ denotes the compact set of $\R^2$ such that $\supp\psi\subset[0,T]\times K$. Thanks to the boundedness property
\eqref{est:u-inf} and the convergence \eqref{conv:dens-strong}, we find
\[
 \left\|I_1\right\|_{L^\infty_T(L^2(K))}\,\leq\,\left\|\rho_n-\rho\right\|_{L^\infty_T(L^2(K))}\,\left\|u_n\right\|_{L^\infty_T(L^\infty)}^2\,\lesssim\,
\left\|\rho_n-\rho\right\|_{L^\infty_T(L^2(K))}\,, 
\]
thus $I_1$ converges to $0$ in $L^\infty_T(L^2_\loc)$, when $n\to+\infty$.
Similarly, for $I_2$ we can write
\[
 \left\|I_2\right\|_{L^\infty_T(L^2(K))}\,\leq\,\rho^*\,\left\|u_n-u\right\|_{L^\infty_T(L^2(K))}\,\left\|u_n\right\|_{L^\infty_T(L^\infty)}\,\lesssim\,
\left\|u_n-u\right\|_{L^\infty_T(L^2(K))}\,\longrightarrow\,0\,,
\]
where the norm convergence follows from \eqref{conv:strong-u}.
In the end, we have proven that
\[
 \rho_n\,u_n\otimes u_n\,\longrightarrow\,\rho\,u\otimes u\qquad\qquad \mbox{ in }\qquad \mc D'\big([0,T]\times\R^2\big)\,.
\]
In turn, this property implies that $\big(\rho,u,\nabla\Pi\big)$ solves the weak formulation of the momentum equation in \eqref{eq:dd-E}, as claimed.

\medbreak
All in all, we have shown that the triplet $\big(\rho,u,\nabla\Pi\big)$ identified in Subsection \ref{ss:limit} is a solution of system \eqref{eq:dd-E}.
In order to complete the proof of Theorem \ref{th:y-exist}, we still have to prove that the vorticity equation \eqref{eq:vort}
holds true in the weak sense: we will deal with this issue in the next subsection.

\subsection{The vorticity equation} \label{ss:vort-eq}

We want to discuss here the validity of the vorticty equation, in order to make more explicit the link of the solutions
we have constructed with the classical Yudovich theory for the homogeneous Euler equations \eqref{eq:hom-E}.

First of all, we have shown above that the triplet $\big(\rho,u,\nabla\Pi\big)$ solves the momentum equation in the weak sense.
Next, we observe that
\[
 \d_t\rho\,=\,-\,\div\big(\rho\,u\big)\,=\,-\,u\cdot\nabla\rho\;\in\,L^\infty\big([0,T];L^{p_0}(\R^2)\big)\,,
\]
owing to the fact that $u$ belongs to $L^\infty_T(W^{1,p_0})$ and $\nabla\rho$ belongs to $L^\infty_T(L^\infty)$.
Therefore, for $\psi\in \mc D\big(\R_+\times\R^2\big)$ as in Paragraph \ref{sss:vel-eq} (so, in particular,
such that $\psi(T)\equiv0$), it makes sense to use $\psi/\rho$ as a test function in the momentum equation.

At the same time, because $u\in L^\infty_T(W^{1,p_0})$ and $\rho\in L^\infty_T(L^\infty)$,
we can apply \tsl{e.g.} Theorem 10.29 of \cite{F-N} and see that $\rho$ is a renormalised solution of the transport equation
$\d_t\rho+u\cdot\nabla\rho=0$. Therefore $a\,:=\,1/\rho$ satisfies the same transport equation, namely
\[
 \d_t\left(\frac{1}{\rho}\right)\,+\,u\cdot\nabla\left(\frac{1}{\rho}\right)\,=\,0\,.
\]
Notice that this could have even been proved directly, by passing to the limit (as done in Paragraph \ref{sss:dens-eq})
in the equation for the $a_n$'s and using the strong convergence property \eqref{conv:inv-dens-str}.

Hence, using the weak form \eqref{eq:weak-mom} of the momentum equation with $\psi/\rho$ as a test-function, and making use of the above transport equation for $1/\rho$,
one immediately sees that the triplet $\big(\rho,u,\nabla\Pi\big)$ also solves, in the weak sense, the equation
\begin{equation} \label{eq:u_second}
 \d_tu\,+\,(u\cdot\nabla) u\,+\,\frac{1}{\rho}\,\nabla\Pi\,=\,0\,.
\end{equation}
At this point, using $\nabla^\perp\phi$ as a test-function in this equation, for $\phi\in \mc D\big(\R_+\times\R^2\big)$, with
$\supp\phi\subset[0,T]\times\R^2$ as in Paragraph \ref{sss:dens-eq}, we see that the vorticity
$\o\,=\,\curl u$ is indeed a weak solution of equation \eqref{eq:vort}.

\medbreak
Before concluding this part, it is worth noticing that the same property could have been proved even more directly, more in the spirit
of the Yudovich proof. Namely, we could have considered the vorticity formulation of the equation for $\big(\rho_n,u_n,\nabla\Pi_n\big)$, \tsl{i.e.} the relation
\begin{equation} \label{eq:vort_n}
 \d_t\o_n\,+\,u_n\cdot\nabla\o_n\,+\,\nabla^\perp\left(\frac{1}{\rho_n}\right)\cdot\nabla\Pi_n\,=\,0\,,
\end{equation}
where $\o_n\,:=\,\curl u_n$ is the vorticity of the velocity field $u_n$, and we could have proved the convergence, for $n\to+\infty$,
of this equation to the weak form of \eqref{eq:vort}.

To this scope, we remark that, with respect to the arguments used in Subsection \ref{ss:limit-eq},
the only problematic term looks to be the density-pressure term, which is non-linear and for which we have established no compactness properties
(nor $\nabla\rho_n$ nor $\nabla\Pi_n$ are compact) so far.
Instead, one has to use a key cancellation property, which implies that
\[
 \nabla^\perp\left(\frac{1}{\rho_n}\right)\cdot\nabla\Pi_n\,=\,\curl\left(\frac{1}{\rho_n}\,\nabla\Pi_n\right)
\]
and integrate by parts: thanks to \eqref{conv:inv-dens-str}, this term in fact converges to the expected limit.

\subsection{Geometric regularity} \label{ss:geom-limit}

In this part, we establish important properties on the geometric quantity $\d_Xu$ which was involved (at the level of approximate solutions) in
assumption \eqref{hyp:geom_unif}.
We start by stating a simple lemma. 

\begin{lemma} \label{l:geom-lim_1}
Under assumptions {\rm \tbf{(A1)}}, {\rm \tbf{(A2)}}, {\rm \tbf{(AE3)}} and \eqref{hyp:geom_unif},
let $\big(\rho,u,\nabla\Pi\big)$ the solution to system \eqref{eq:dd-E} constructed above. Set $X\,:=\,\nabla^\perp\rho$.

Then $\d_Xu$ belongs to $\mc M\big([0,T];L^\infty(\R^2)\big)\,\cap\,L^\infty\big([0,T];L^{p_0}(\R^2\big)$, where $\mc M(I)$ denotes the set of
Radon measures over some interval $I\subset\R$. Moreover, one has the weak-$*$ convergence
\[
 \d_{X_n}u_n\,\stackrel{*}{\wtend}\,\d_Xu\qquad\qquad \mbox{ in }\qquad L^\infty\big([0,T];L^{p_0}(\R^2\big)\,. 
\]
If we replace assumption {\rm \tbf{(AE3)}} with {\rm \tbf{(AU3)}}, namely if in addition one has $\o_0\,\in\,L^\infty(\R^2)$,
then $\d_Xu$ belongs to $L^\infty\big([0,T];L^q(\R^2)\big)$ for any $q\in[p_0,+\infty[\,$ and the convergence holds true also
in the weak-$*$ topology of those spaces. 
\end{lemma}

\begin{proof}
We start the proof by observing that the target quantity $\d_Xu\,=\,(X\cdot\nabla)u$ belongs to $L^\infty_T(L^{p_0})$, owing to the fact that
$X\,=\,\nabla^\perp\rho$ is bounded over $[0,T]\times\R^2$ and $\nabla u$ belongs to $L^\infty_T(L^{p_0})$: recall \eqref{conv:dens-weak}
and \eqref{conv:u-more}, respectively.

Next, we want to show that $\big(\d_{X_n}u_n\big)_{n\in\N}$ converges to $\d_Xu$ in the weak-$*$ topology of $L^\infty_T(L^{p_0})$. Here, we face a similar
difficulty as the one discussed after \eqref{eq:vort_n}, as nor the $X_n$'s, nor the $\nabla u_n$'s are compact in any sense. However, also in this case
we can take advantage of a fundamental cancellation: owing to the fact that $\div X_n=0$ for any $n\in\N$, we can write
\[
 \d_{X_n}u_n\,=\,\div\big(X_n\otimes u_n\big)\,=\,\d_1\big(X_n^1\,u_n\big)\,+\,\d_2\big(X_n^2\,u_n\big)\,.
\]
Thus, we can use the weak convergence \eqref{conv:dens-weak} and the strong convergence \eqref{conv:strong-u} to deduce that
\[
 \d_{X_n}u_n\,\longrightarrow\,\d_Xu\qquad\qquad \mbox{ in }\qquad \mc D'\big([0,T]\times\R^2\big)\,.
\]
On the other hand, it follows from \eqref{est:unif_D-rho} and \eqref{est:u_W^1} that $\big(\d_{X_n}u_n\big)_{n\in\N}\,\sqsubset\,L^\infty_T(L^{p_0})$,
so this sequence must converge (up to a suitable extraction) to some vector field $\xi\,\in\,L^\infty_T(L^{p_0})$ in the weak-$*$ topology of that space.
Now, the uniqueness of the weak limit implies that $\xi\,\equiv\,\d_Xu$ almost everywhere on $[0,T]\times\R^2$, which in turn also yields the convergence
$\d_{X_n}u_n\,\stackrel{*}{\wtend}\,\d_Xu$ in $L^\infty_T(L^{p_0})$.

The assertion $\d_Xu\,\in\,\mc M\big([0,T];L^\infty(\R^2)\big)$ follows from the uniform boundedness property
$\big(\d_{X_n}u_n\big)_{n\in\N}\,\sqsubset\,L^1_T(L^\infty)$, instrinsic of assumption \eqref{hyp:geom_unif}.

Finally, the last sentence of the statement is a straightforward consequence of \eqref{est:geom-unif},
combined with the uniqueness of the weak limit. This completes the proof of the lemma.
\end{proof}

It is important to remark that, in the previous lemma, we fail to get the property $\d_Xu\,\in\,L^1_T(L^\infty)$, which is expected to play a key role in the proof
of uniqueness though. See also the comments in Remark \ref{r:geom_unif_to-lim} above.

However, we now show that, replacing assumption \eqref{hyp:geom_unif} with \eqref{cond:strong_geom},
namely further assuming (with respect to the statement of Lemma \ref{l:geom-lim_1})
that $\big(\d_{X_n}u_n\big)_{n\in\N}\,\sqsubset\,L^{p_*}_T(L^\infty)$, for some $p_*>1$,
we can establish the sought property $\d_Xu\,\in\,L^1_T(L^\infty)$.
This will justify the claim contained in Remark \ref{r:geom-limit}.

\begin{prop} \label{p:geom-lim_tot}
Assume {\rm \tbf{(A1)}}, {\rm \tbf{(A2)}}, {\rm \tbf{(AE3)}} and  \eqref{cond:strong_geom}.
Let $\big(\rho,u,\nabla\Pi\big)$ the solution to system \eqref{eq:dd-E} constructed above, and define $X\,:=\,\nabla^\perp\rho$.

Then, the geometric quantity $\d_Xu$ belongs to $L^{p_*}\big([0,T];L^{\infty}(\R^2)\big)$, with the same $p_*>1$ appearing in condition \eqref{cond:strong_geom}.
In particular, one has the property $\d_Xu\,\in\,L^1\big([0,T];L^\infty(\R^2)\big)$.
\end{prop}

\begin{proof}
The claimed property is a direct consequence of \eqref{est:geom_unif-str}, which implies that
$\big(\d_{X_n}u_n\big)_{n\in\N}$ is uniformly bounded in the space $L^{p_*}_T(L^\infty)$,
combined with the uniqueness of the weak limit.
\end{proof}

\section{Uniqueness of Yudovich-type solutions} \label{s:uniqueness}

In this section, we tackle the question of uniqueness of solutions to system \eqref{eq:dd-E} at the given level of regularity.
More precisely, we first prove Theorem \ref{th:y-exist_2} (which is an improved version of Theorem \ref{th:y-exist}),
together with further regularity properties of the constructed solutions
under the additional assumption that the initial vorticity is also bounded.
This will be the matter of Subsection \ref{ss:further}.
Then, in Subsection \ref{ss:unique} we move to the proof of the uniqueness statement, that is Theorem \ref{th:yudovich},
which storngly relies on the property $\o_0\in L^\infty$ and on the geometric assumption \eqref{hyp:geom_lim}.

Throughout this section, we replace assumption \tbf{(AE3)} with the stronger \tbf{(AU3)}, which postulates in addition that
\[
 \o_0\,\in\,L^\infty(\R^2)\,.
\]
As already observed in Remark \ref{r:o-limit}, under this assumption, 
the construction given in Section \ref{s:exist} allows us to deduce the additional property
\begin{equation} \label{eq:o-bounded}
 \o\,\in\,L^\infty\big([0,T];L^\infty(\R^2)\big)\,.
\end{equation}
As it was already the case for the classical Yudovich theory, this is a key information
for establishing suitable regularity of the velocity field, in particular for obtaining the existence of a unique flow associated to it,
and for proving uniqueness.

%

\subsection{Further properties of the solutions} \label{ss:further}

In the present subsection, we establish further properties on the constructed solution $\big(\rho,u,\nabla\Pi\big)$.
These are key pieces of information in order to prove (in Subsection \ref{ss:unique}) the uniqueness of Yudovich-type solutions.
Differently from what done in Subsections \ref{ss:further-bounds} and \ref{ss:geom-limit},
these properties are not inferred from the convergence of the family of approximate solutions, but rather derived directly
from the equations, by using the additional bound \eqref{eq:o-bounded}.

In doing so, however, we will also complete the proof of Theorem \ref{th:y-exist_2}.
Notice that the boundedness of the vorticity $\o$ has just been established in \eqref{eq:o-bounded}.
Hence, thanks to the arguments developed in Sections \ref{s:uniform} and \ref{s:exist}, in order to complete the proof of Theorem \ref{th:y-exist_2}
we only need to show the following facts: 
\begin{itemize}
 \item the boundedness of $\nabla\Pi$, claimed in item (d) of the statement;
 \item the property $u\in L^\infty\big([0,T];\mc Z(\R^2)\big)$, claimed in item (b) of the statement.
\end{itemize}


Let us start the discussion by recalling a couple of basic definitions. First of all, we give the precise definition of the Zygmund class $\mc Z$
mentioned in the statement of Theorem \ref{th:yudovich}.

\begin{defin} \label{def:Z}
A function $f\in L^\infty(\R^d)$ is said to be \emph{Zygmund} continuous, and we write $f\in \mc Z(\R^d)$, if the quantity
\[
|f|_{\mc Z}\,:=\,\sup_{z,y\in\R^d,\,0<|y|<1}
\left(\frac{\left|f(z+y)\,+\,f(z-y)\,-\,2\,f(z)\right|}{|y|}\right)\,<\,+\infty\,.
\]
We define the Zygmund norm as $\|f\|_{\mc Z}\,:=\,\|f\|_{L^\infty}\,+\,|f|_{\mc Z}$.
\end{defin}

Observe that $\mc Z$ coincides with the Besov space $B^1_{\infty,\infty}$ (see Chapter 2 of \cite{BCD} for a proof).
By Proposition 2.107 of \cite{BCD}, one has that $\mc Z(\R^d)$ is embedded in the space $LL(\R^d)$ of
log-Lipschitz functions, whose definition is recalled below.

\begin{defin} \label{def:LL}
A function $f\in L^\infty(\R^d)$ is said to be \emph{log-Lipschitz} continuous, and we write $f\in LL(\R^d)$, if the quantity
\[
|f|_{LL}\,:=\,\sup_{z,y\in\R^d,\,0<|y|<1}
\left(\frac{\left|f(z+y)\,-\,f(z)\right|}{|y|\,\log\left(1\,+\,\frac{1}{|y|}\right)}\right)\,<\,+\infty\,.
\]
We define the log-Lipschitz norm as $\|f\|_{LL}\,:=\,\|f\|_{L^\infty}\,+\,|f|_{LL}$.
\end{defin}

We want to recall that vector fields which are $L^1_T(LL)$ admit a uniqueley defined flow, see \tsl{e.g.} Chapter 3 of \cite{BCD} for details.

\medbreak
After these preliminaries, we can state the following static result, which in particular implies the validity of the property claimed
in item (b) of Theorem \ref{th:y-exist_2}.
In this result, we will use the fact $u\in L^\infty(\R^2)$; besides, this will dispense us from imposing
a low index integrability condition on the vorticity field  (in our
statement, we only take $p_0>2$), which is instead required in the classical theory, see \tsl{e.g.} Lemma 8.1 of \cite{Maj-Bert}.
Notice that, with respect to the previous reference, we slightly improve the regularity
of the velocity field, by showing that $u$ is Zygmund continuous.

\begin{prop} \label{p:u-Zyg}
Let $u$ be a vector field in $L^\infty(\R^2)$ such that $\div u=0$. Let $\o\,=\,\curl u\,=\,\d_1u^1-\d_2u^1$ be its vorticity. Assume that
$\o\in L^\infty(\R^2)$.

Then $u\in \mc Z(\R^2)$.
\end{prop}

\begin{proof}
The proof of the previous property, which by the way holds true in any space dimension $d\geq2$, would be easy by working with the Besov characterisation of the Zygmund class, that is $\mc Z(\R^d)\equiv B^1_{\infty,\infty}(\R^d)$. However, in order to avoid any use of the Littlewood-Paley theory here,
we prove the result by directly working on the Biot-Savart law \eqref{eq:BS-law} and performing potential estimates, in the spirit of \cite{Maj-Bert}.
For this reason, we stick to the two-dimensional case in our argument.

So, given a unitary vector $e\in \mbb{S}^1$ and $x\in\R^2$, we look for an estimate of the type
\begin{equation} \label{est:Z-to-do}
 \Big|u(x+he)\,+\,u(x-he)\,-\,2\,u(x)\Big|\,\leq\,C\,|h|\,,
\end{equation}
for $0<h<1$ and for a suitable constant $C>0$ independent of $x$, $e$ and $h$.
We now use the Biot-Savart law to express the previous quantity in a different way.

According to \eqref{eq:BS-law}, the Biot-Savart law writes
\[
 u\,=\,\frac{1}{2\pi}\,\int_{\R^2} K(x-y)\,\o(y)\,\dd y\,,\qquad\qquad \mbox{ with }\qquad  K(z)\,:=\,\frac{1}{|z|^2}\,z^\perp\,.
\]
With this notation, and using that $\o\in L^\infty(\R^2)$, the left-hand side of \eqref{est:Z-to-do} can be bounded in the following way:
\begin{align*}
 &\Big|u(x+he)\,+\,u(x-he)\,-\,2\,u(x)\Big| \\
&\qquad\qquad\qquad\leq\,\left\|\o\right\|_{L^\infty}\,\int_{\R^2}\Big|K(x+he-y)\,+\,K(x-he-y)\,-\,2\,K(x-y)\Big|\,\dd y\,.
\end{align*}
Observing that the kernel $K$ associated to the Biot-Savart law is smooth far from the origin, we split the integral on the right-hand side of the previous equality
in the two regions 
\[
 D_1\,:=\,\big\{|x-y|\geq 2r\big\}\qquad\qquad \mbox{ and }\qquad\qquad
 D_2\,:=\,\big\{|x-y|< 2r\big\}\,,
\]
with $r>0$ to be chosen later on (in fact, we will see that $r=h$ is the optimal choice).

For $r$ sufficiently large, namely $r\geq h$, 
we have that $x\pm he-y\geq r$ on the set $D_1$.
At the same time, using the smoothness of $K$ far from the origin, by a Taylor expansion up to the fourth order it is a routine matter to get
\[
 \Big|K(x+he-y)\,+\,K(x-he-y)\,-\,2\,K(x-y)\Big|\,\leq\,h^2\,\Big|\nabla^2K(x-y)\Big|\,+\,C\,h^4\,\Big|\nabla^4K(\xi)\Big|\,,
\]
for a suitable $\xi$ belonging to the segment joining the points $x-he-y$ and $x+he-y$.
Therefore, using that $\big|\nabla^nK(z)\big|\,\lesssim\,|z|^{-(n+1)}$ and that $|\xi|\,\geq\,|x-y|-h\,\geq\,r$, passing in polar coordinates one easily see that
\begin{align}
\label{est:D_1}
&\int_{D_1}\Big|K(x+he-y)\,+\,K(x-he-y)\,-\,2\,K(x-y)\Big|\,\dd y \\
\nonumber
&\qquad\qquad\qquad\qquad\qquad\qquad\qquad \lesssim\,h^2\,\int^{+\infty}_{2r}\frac{1}{\s^2}\,\dd\s\,+\,
h^4\,\int^{+\infty}_{2r}\frac{1}{(\s-h)^5}\,\s\,\dd\s \\
\nonumber
&\qquad\qquad\qquad\qquad\qquad\qquad\qquad \lesssim\,\frac{h^2}{r}\,+\,\frac{h^4}{r^3}\,+\,\frac{h^5}{r^4}\,.
\end{align}

Let us now focus on the integral over $D_2$. In that region, we can use that the kernel $K$ is integrable close to the origin.
Observe that, on $D_2$, we have $|x\pm he -y|\leq|x-y|+h\leq 2r+h$ by triangular inequality. Thus, by translation invariance of the integral,
we are led to the following estimate:
\begin{align}
\label{est:D_2}
&\int_{D_2}\Big|K(x+he-y)\,+\,K(x-he-y)\,-\,2\,K(x-y)\Big|\,\dd y \\
\nonumber
&\qquad\qquad\qquad\qquad\qquad
\lesssim\,\int_{\big\{|z|\leq 2r+h\big\}}\frac{1}{|z|}\,\dd z\,+\,\int_{\big\{|z|\leq 2r\big\}}\frac{1}{|z|}\,\dd z\;\lesssim\;r\,+\,h\,.
\end{align}

Putting estimates \eqref{est:D_1} and \eqref{est:D_2} together and (as anticipated above) choosing $r=h$, we finally obtain the sought inequality
\eqref{est:Z-to-do}, which in turn implies that $u$ belongs to the Zygmund class. This completes the proof of the proposition.
\end{proof}

As already pointed out before, from Proposition \ref{p:u-Zyg} we deduce the following result, which yields, in particular,
item (b) appearing in the statement of Theorem \ref{th:yudovich}.

\begin{cor} \label{c:flow}
 Let $u$ be a divergence-free velocity field satisfying the property of item {\rm (ii)} in Theorem \ref{th:y-exist} and such that
its vorticity $\o\,:=\,\curl u$ verifies \eqref{eq:o-bounded}.

Then $u$ belongs to the space
$L^\infty\big([0,T];\mc Z(\R^2)\big)$. In particular, $u\,\in\,L^\infty\big([0,T];LL(\R^2)\big)$,
hence it admits a unique flow map.
\end{cor}

Next, we deal with the regularity of the pressure gradient: the following result establishes the property $\nabla\Pi\in L^\infty$.
\begin{prop} \label{p:press-inf}
Let $\big(\rho,u,\nabla\Pi)$ be a solution to system \eqref{eq:dd-E} enjoying the properties listed in
items {\rm(i)} to {\rm(iv)} of Theorem \ref{th:y-exist}, plus \eqref{eq:o-bounded}.

Then one has $\nabla\Pi\,\in\,L^\infty\big([0,T]\times \R^2\big)$.
\end{prop}

\begin{proof}
We will derive the claimed property in two steps. First of all, by applying the divergence operator to \eqref{eq:u_second}, we see
that $\nabla\Pi$ solves, in the weak sense, the equation
\begin{equation} \label{eq:pressure}
 -\,\div\left(\frac{1}{\rho}\,\nabla\Pi\right)\,=\,\div\Big((u\cdot\nabla) u\Big)\,.
\end{equation}
By developing the derivatives on the left and on the right (more rigorously, by testing the previous equation on test-functions of the form $\rho\,\phi$,
which is possible owing to the assumed regularity of the solution $\big(\rho,u,\nabla\Pi\big)$) and using the divergence-free condition
over $u$, we find an equation (satisfied in the weak form) for $\Delta\Pi$, namely
\begin{align}
\label{eq:Delta-Pi}
-\,\Delta\Pi\,&=\,-\,\frac{1}{\rho}\,\nabla\rho\cdot\nabla\Pi\,+\,\rho\,\nabla u:\nabla u\,. 
\end{align}
As all the functions are $L^\infty$ in time, let us drop the dependence on the time variable in the argument below.

We observe that, by assumption, $\nabla\Pi\in L^2$, while both $1/\rho$ and $\nabla\rho$ belong to $L^\infty$. So the first term on the right-hand side
of relation \eqref{eq:Delta-Pi} belongs to $L^2$. Next, because $\o\in L^{p_0}\cap L^\infty$, by use of the Calder\'on-Zygmund theory we know that
\begin{equation} \label{eq:integrab_Du}
 \nabla u \,\in\,\bigcap_{q\in [p_0,+\infty[\,}L^q\,,
\end{equation}
which in particular implies that
\[
 \nabla u:\nabla u\,\in\,\bigcap_{q\in[\frac{p_0}{2},+\infty[\,}L^q\,.
\]
In turn, owing to the assumption\footnote{This is the only place of our analysis where we need the condition $p_0\leq4$.} $p_0\leq4$,
this latter property yields that also the second term $\nabla u:\nabla u$ appearing in \eqref{eq:Delta-Pi} belongs to $L^2$.

Putting all those pieces of information together, we get the property $\Delta\Pi\in L^2$, which implies (by Sobolev embeddings) that
$\nabla\Pi\in L^q$ for all $q\in[2,+\infty[\,$.
Thanks to this fact, we can go back to equation \eqref{eq:Delta-Pi} and deduce that there exists some finite $p>d=2$ such that $\Delta\Pi\in L^p$,
whence $\nabla \Pi\in W^{1,p}\hookrightarrow L^\infty$, as claimed.
\end{proof}

Proposition \ref{p:press-inf} completes the proof of Theorem \ref{th:y-exist_2}.
Thus, what is left is to show the proof of Theorem \ref{th:yudovich} concerning uniqueness of Yudovich-type solutions which, in addition, satisfy
the geometric regularity property \eqref{hyp:geom_lim}.
For this, we need an auxiliary result, concerning the time regularity of the velocity fields.
It is worth noticing that this result does not rely on the property \eqref{eq:o-bounded}; in particular, it holds true for any solution satisfying
the properties listed in Theorem \ref{th:y-exist}.

\begin{lemma} \label{l:d_tu}
Let $\big(\rho,u,\nabla\Pi)$ be a solution to system \eqref{eq:dd-E} enjoying the properties listed in
items from {\rm(i)} to {\rm(iv)} of Theorem \ref{th:y-exist}. 

Then one has $u\,\in\, W^{1,\infty}\big([0,T];L^2(\R^2)\big)$.
\end{lemma}

\begin{proof}
Thanks to the fact that $\rho(t,x)\geq\rho_*>0$ for any $(t,x)\in[0,T]\times\R^2$, we can write the momentum equation as
in \eqref{eq:u_second}, namely
\[ 
 \d_tu\,=\,-\,u\cdot\nabla u\,-\,\frac{1}{\rho}\,\nabla\Pi\,.
\] 
By using the same argument employed for studying equation \eqref{eq:d_tu_n} above, we see that the right-hand side of the previous relation
belongs to $L^\infty\big([0,T];L^2(\R^2)\big)$. This concludes the proof.
\end{proof}


\subsection{Proof of uniqueness} \label{ss:unique}

In this section we prove that, under an additional geometric regularity condition contained in assumption \eqref{hyp:geom_lim},
the previously constructed solutions are in fact unique.
A fundamental assumption is that the two (supposed \tsl{a priori} distinct) solutions pertain to the same initial datum. Therefore,
our result is \emph{not} a stability result. Despite that, we will often speak about ``stability estimates'' just to mean that we will perform
estimates for the difference of the two solutions.

Roughly speaking, our proof mimicks the proof of uniqueness in the Yudovich theorem, as presented in Chapter 8 of
\cite{Maj-Bert}. More precisely, we will perform $L^2$ stability estimates for the velocity field. in the non-homogeneous setting, however,
this will invoke a $L^2$ stability estimate for the density as well, together with a suitable estimate on the pressure function.
Let us observe that previous $L^2$ stability results for system \eqref{eq:dd-E}
(see \tsl{e.g.} Proposition 8 of \cite{D_2010}) do not apply here, as they all rely on a $L^1_T(L^\infty)$ control on $\nabla u$, which is
out of reach within our theory.

\medbreak
Thus, let us assume to have an initial datum $\big(\rho_0,u_0\big)$ satisfying assumptions \tbf{(A1)}, \tbf{(A2)} and \tbf{(AU3)}.
Let $\big(\rho_1,u_1,\nabla\Pi_1\big)$ and
$\big(\rho_2,u_2,\nabla\Pi_2\big)$ two solutions related to it, defined on $[0,T]\times\R^2$ for some positive time $T>0$,
such that (as already proven) the properties listed in items from (a) to (d) of Theorem \ref{th:y-exist_2} hold true. Moreover, we
assume that they both verify \eqref{hyp:geom_lim}.
Notice that also Lemma \ref{l:d_tu} holds true for such solutions.

\subsubsection{\tsl{A priori} estimates} \label{est:a-priori}
First of all, we need to establish some \tsl{a priori} bounds for the two triplets $\big(\rho_j,u_j,\nabla\Pi_j\big)_{j=1,2}$.
For the sake of better readability, let us drop the index $j$ in the computations below.

As postulated by condition \eqref{hyp:geom_lim}, both solutions satisfy the additional geometric
property $\d_Xu\,\in\,L^1_T(L^\infty)$, where $X=\nabla^\perp\rho$, as usual in this paper.
We are now going to establish suitable quantitative \tsl{a priori} estimates on the time interval $[0,T]$, in terms of the quantity
\begin{equation} \label{def:K_0}
M_0\,:=\,\int^T_0\left\|\d_{X(t)}u(t)\right\|_{L^\infty}\,\dt\,.
\end{equation}
Most of the arguments are similar to the computations developed in Section \ref{s:uniform}. Therefore, we will essentially list the needed properties,
while being a little bit sketchy about the proofs.

Our estimates will be derived from transport equations satisfied by the different quantities. The properties established in Proposition \ref{p:u-Zyg} and Corollary
\ref{c:flow} are thus fundamental to rigorously justify our computations.

\subsubsection*{The density}

We start by considering the mass equation. As seen in Subsection \ref{ss:limit}, the density $\rho$ belongs to $W^{1,\infty}\big([0,T]\times\R^2\big)$
and (by Proposition \ref{p:u-Zyg}) the velocity field $u$ to $L^\infty_T(\mc Z)$. From this, we easily deduce that
\begin{equation} \label{apr_est:dens}
\forall\,(t,x)\in[0,T]\times\R^2\,,\qquad\qquad 0<\rho_*\leq\rho(t,x)\leq\rho^*\,,
\end{equation}
where the two constants $\rho_*$ and $\rho^*$ are the ones identified in assumption \tbf{(A1)}.

Next, by direct differentiation (which is well-justified in our setting), we see that the vector field $X\,:=\,\nabla^\perp\rho$ satisfies the analogue
of equation \eqref{eq:X}, namely
\[
 \d_tX\,+\,u\cdot\nabla X\,=\,\d_Xu\,.
\]
From this relation and transport estimates, we infer that
\begin{equation} \label{apr_est:X}
\sup_{t\in[0,T]}\left\|\nabla\rho(t)\right\|_{L^\infty}\,\leq\,\left\|\nabla\rho_0\right\|_{L^\infty}\,+\,M_0\,,
\end{equation}
where $M_0$ is the (finite) value of the integral quantity in \eqref{def:K_0}.

\subsubsection*{The velocity field}

Let us now focus on the bounds for the velocity field $u$. First of all, notice that, by the same token used for deriving \eqref{est:bilin-2}, one has
the property $\rho (u\cdot\nabla) u\,\in\,L^\infty_T(L^2)$. Similarly, we know that $\nabla\Pi\in L^\infty_T(L^2)$. Then, thanks also to Lemma \ref{l:d_tu},
we can take the $L^2$ scalar product of the momentum equation by $u$ and get that
\begin{equation} \label{eq:d_t-kinetic}
 \frac{1}{2}\,\frac{\dd}{\dt}\int_{\R^2}\rho\,|u|^2\,\dx\,=\,0\,.
\end{equation}
Here, we have used the divergence-free property on $u$ to get rid of the pressure gradient. In addition, key cancellations of the time derivative and the 
(non-linear) transport term are obtained by using the momentum equation for $\rho$, tested against $\phi=|u|^2$. Notice that this operation
is well-justified, since (similarly as for getting \eqref{est:bilin-2} and thanks to Lemma \ref{l:d_tu})
one easily checks that $\phi\in L^\infty\big([0,T];H^1(\R^2)\big)\cap W^{1,\infty}\big([0,T];L^1(\R^2)\big)$.
After an integration in time, and making use of \eqref{apr_est:dens}, we find
\begin{equation} \label{apr_est:u-2}
\sup_{t\in[0,T]}\left\|u(t)\right\|_{L^2}\,\lesssim\,\left\|u_0\right\|_{L^2}\,.
\end{equation}

We can establish further bounds for the velocity fields and its vorticity by arguing precisely as done in Section \ref{s:uniform}. By
the analogue of estimates \eqref{est:eta-L_first} and \eqref{est:u-inf_prelim}, we easily deduce that
\begin{equation} \label{apr_est:eta}
 \sup_{t\in[0,T]}\left\|\eta(t)\right\|_{L^{p_0}\cap L^\infty}\,\lesssim\,\left\|\eta_0\right\|_{L^{p_0}\cap L^\infty}\,\exp\Big(C\,T\,+\,C\,M_0\Big)\,,
\end{equation}
where, as usual, we have defined $\eta\,:=\,\curl\big(\rho\,u\big)$ and where $C$ is a universal constant, only depending on the constants $\rho_*$,
$\rho^*$ and $\|u_0\|_{L^2}$. As already observed in Section \ref{s:uniform}, our assumptions imply that the $L^{p_0}\cap L^\infty$ norm of
$\eta_0\,=\,\curl\big(\rho_0\,u_0\big)$ is finite.

Since the following bounds will be more intricate to get, from now on we will not keep track of the precise value of the constants occurring in our
computations. We will generically denote them by $C$, tacitly meaning that they may depend on suitable norms of the initial data, as well as
on the fixed time $T>0$ and on the quantity $M_0$ defined in \eqref{def:K_0}.

Let us resume with the derivation of \tsl{a priori} estimates.
Inequality \eqref{apr_est:eta}, together with \eqref{apr_est:dens}, \eqref{apr_est:X} and \eqref{apr_est:u-2}, implies the following additional bounds for $u$
and its vorticity $\o\,=\,\curl u$, for a suitable (implicit) constant $C$:
\begin{align}
 \label{apr_est:u-inf}
 \sup_{t\in[0,T]}\left\|u(t)\right\|_{L^\infty}\,&\lesssim\,1\,, \\
\label{apr_est:omega}
\sup_{t\in[0,T]}\left\|\o(t)\right\|_{L^{p_0}\cap L^\infty}\,&\lesssim\,1\,,
\end{align}
where the latter bound is a direct consequence of the analogue of relation \eqref{est:o_n-eta_n}.

\subsubsection*{The pressure gradient}

Next, we need to establish convenient bounds for the pressure gradient. Recall that $\nabla\Pi$ solves the elliptic equation
\eqref{eq:pressure}.
Applying Lemma 2 of \cite{D_2010} and arguing as for obtaining \eqref{est:bilin-2}, we find
\begin{align*}
\frac{1}{\rho^*}\,\left\|\nabla\Pi\right\|_{L^2}\,&\lesssim\,\left\|(u\cdot\nabla) u\right\|_{L^2}\,
\lesssim\,\left\|u\right\|_{L^2\cap L^\infty}\,\left\|\nabla u\right\|_{L^{p_0}} \\
&\lesssim\,\frac{p_0^2}{p_0-1}\,\left\|u\right\|_{L^2\cap L^\infty}\,\left\|\o\right\|_{L^{p_0}}\,.
\end{align*}
Notice that, for passing from the first to the second line, we have used the optimal bound provided by Calder\'on-Zygmund theory to estimate
$\nabla u$ by the corresponding Lebesgue norm of the vorticity $\o$. Since $p_0$ is fixed here, it can be treated as a constant. Then, thanks to
\eqref{apr_est:u-2}, \eqref{apr_est:u-inf} and \eqref{apr_est:omega}, the above estimate yields
\begin{equation} \label{apr_est:Pi-2}
 \sup_{t\in[0,T]}\left\|\nabla\Pi(t)\right\|_{L^2}\,\lesssim\,1\,.
\end{equation}

This bound is not enough for our scopes: we need a control on the $L^q$ norm of $\nabla\Pi$, for $q$ large. Owing to Proposition \ref{p:press-inf},
we could in fact estimate its $L^\infty$ norm, but this bound would be more involved (and so would be the argument of the proof) and not really necessary
for deriving uniqueness (see the next paragraph). Instead, we simply establish a $L^2$ estimate on $\Delta\Pi$. For this, we use
relation \eqref{eq:Delta-Pi} to infer that
\begin{align*}
 \left\|\Delta\Pi\right\|_{L^2}\,&\lesssim\,\left\|\nabla\rho\right\|_{L^\infty}\,\left\|\nabla\Pi\right\|_{L^2}\,+\,
\rho^*\,\left\|\nabla u:\nabla u\right\|_{L^2} \\
&\lesssim\,\left\|\nabla\rho\right\|_{L^\infty}\,\left\|\nabla\Pi\right\|_{L^2}\,+\,
\left\|\nabla u\right\|_{L^{p_1}}\,\left\|\nabla u\right\|_{L^{p_0}}\,,
\end{align*}
where we have defined the integrability index $p_1\in\,]2,+\infty[\,$ such that $1/p_0\,+\,1/p_1\,=\,1/2$.
Notice that, thanks to the assumption $p_0\leq 4$, we have that $p_1\geq p_0$.

For the estimate of the first term on the last line, we can argue as for establishing \eqref{apr_est:Pi-2}: using also \eqref{apr_est:X}, we find that this term
is bounded by a constant, uniformly on $[0,T]$. So, let us focus on the second term only, and more precisely on the bound of $\nabla u$ in $L^{p_1}$.
From Calder\'on-Zygmund theory and the analogue of the second bound in \eqref{est:o_n-eta_n}, thanks to the fact that $p_1\geq p_0$ we may estimate
\begin{align*}
 \left\|\nabla u\right\|_{L^{p_1}}\,\lesssim\,\left\|\o\right\|_{L^{p_1}}\,&\lesssim\,\left\|\eta\right\|_{L^{p_1}}\,+\,\left\|u\cdot\nabla\rho\right\|_{L^{p_1}} \\
&\lesssim\,\left\|\eta\right\|_{L^{p_0}\cap L^\infty}\,+\,\left\|u\right\|_{L^{2}\cap L^\infty}\,\left\|\nabla\rho\right\|_{L^\infty}\,\lesssim\,1\,.
\end{align*}
Notice that the above (implicit) multiplicative constant depends on $p_1$, whence on $p_0$, which however is fixed in this argument.
All in all, we have proved that
\begin{equation} \label{apr_est:Delta-Pi}
 \left\|\Delta\Pi\right\|_{L^2}\,\lesssim\,1\,.
\end{equation}

As already pointed out, we observe that $H^1(\R^2)$ fails to embed in $L^\infty(\R^2)$, but it embeds into $L^p(\R^2)$ for any finite $p\geq2$. In addition,
from inequality (1.41) of \cite{BCD}, we know that
\[
\forall\,p\geq2\,,\qquad\qquad \left\|\nabla\Pi\right\|_{L^p}\,\leq\,C\,\sqrt{p}\,\left\|\nabla\Pi\right\|_{H^1}\,.
\]
Thus, putting all these pieces of information together and using \eqref{apr_est:Pi-2} and \eqref{apr_est:Delta-Pi}, we infer the inequality
\begin{equation} \label{apr_est:Pi-p}
\forall\,p\geq2\,,\qquad\qquad
 \sup_{t\in[0,T]}\left\|\nabla\Pi(t)\right\|_{L^p}\,\lesssim\,\sqrt{p}\,.
\end{equation}

\subsubsection{Stability estimates and uniqueness} \label{est:stab}

We can now prove the uniqueness of solutions, completing in this way the proof of Theorem \ref{th:yudovich}. The argument will be similar to the original one employed
in the classical Yudovich theorem, see \tsl{e.g.} Chapter 8 of \cite{Maj-Bert}.

In order to prove uniqueness, let us define
\[
 \delta\rho\,:=\,\rho_1\,-\,\rho_2\,,\qquad\qquad \delta u\,:=\,u_1\,-\,u_2\,,\qquad\qquad \nabla\delta\Pi\,=\,\nabla\Pi_1\,-\,\nabla\Pi_2\,.
\]
Simple computations yield that those quantities solve the system
\begin{equation} \label{eq:delta-eq}
 \left\{\begin{array}{l}
         \d_t\de\rho\,+\,u_1\cdot\nabla\de\rho\,=\,-\,\de u\cdot\nabla \rho_2 \\[1ex]
         \rho_1\,\d_t\de u\,+\,\rho_1\,(u_1\cdot\nabla)\de u\,+\,\nabla\de\Pi\,=\,\dfrac{\de\rho}{\rho_2}\,\nabla\Pi_2\,-\,\rho_1\,(\de u\cdot\nabla) u_2 \\[2ex]
         \div\de u\,=\,0
        \end{array}
\right.
\end{equation}
with zero initial datum (by assumption).

Similarly to what done for establishing \eqref{est:bilin-2}, and owing to Lemma \ref{l:d_tu} and the property $\nabla\Pi\in L^\infty_T(L^2)$, we see that all the terms
entering into play in the second equation above belong to $L^\infty_T(L^2)$. So, we can take the $L^2$-scalar
product of that equation by $\de u$: by using the divergence-free condition on the velocity fields and their difference, we find
\begin{align}
\label{est:du-L^2}
\frac{1}{2}\,\frac{\dd}{\dt}\int_{\R^2}\rho_1\,\left|\de u\right|^2\,\dx\,=\,\int_{\R^2}\dfrac{\de\rho}{\rho_2}\,\nabla\Pi_2\cdot\de u\,\dx\,-\,
\int_{\R^2}\rho_1\,(\de u\cdot\nabla) u_2\cdot\de u\,\dx\,.
\end{align}
Observe that, in obtaining the previous relation, we have used the mass equation for $\rho_1$, namely the relation
\[
 \d_t\rho_1\,+\,\div\big(\rho_1\,u_1\big)\,=\,0\,,
\]
in the weak form. Testing it against $\phi\,=\,\big|\de u\big|^2$ is a well-justified operation, since, precisely as done for obtaining \eqref{eq:d_t-kinetic},
one has that $\phi\in L^\infty\big([0,T];H^1(\R^2)\big)\,\cap\,W^{1,\infty}\big([0,T];L^1(\R^2)\big)$, where the latter property
follows from Lemma \ref{l:d_tu}.

Relation \eqref{est:du-L^2} asks for a $L^2$ bound on $\de\rho$. Despite each $\rho_j$ only belongs to $L^\infty$ in space, it follows from
the first equation in \eqref{eq:delta-eq} that $\de\rho$ does belong to $L^2$. Indeed, this quantity is transported by the
vector field $u_1$, which is log-Lipschitz and, as such, admits a unique flow which, in addition, is measure-preserving (owing to the divergence-free
constraint over $u_1$); moreover, the initial datum belongs to $L^2$ (in fact, it is identically $0$)
and the forcing term $-\de u\cdot\nabla\rho_2$ is in $L^\infty_T(L^2)$. Hence, $\de\rho\in L^\infty_T(L^2)$, as claimed.
Furthermore, also $u_1\cdot\nabla\de\rho$ belongs to $L^\infty_T(L^2)$, as $u_1$ belongs to this space and $\nabla\de\rho$ is bounded.
Therefore, it makes sense to take the $L^2$ scalar product of that equation by $\de\rho$: we find
\begin{equation} \label{est:drho-L^2}
 \frac{1}{2}\,\frac{\dd}{\dt}\int_{\R^2}\left|\de\rho\right|^2\,\dx\,=\,-\,\int_{\R^2}\de u\cdot\nabla\rho_2\,\de\rho\,\dx\,.
\end{equation}
We point out that relies on the cancellation
\[
 \int^t_0\int_{\R^2}u_1\cdot\nabla\left|\de\rho\right|^2\,\dx\,=\,0\,,
\]
which holds true whenever this expression makes sense, as $\div u_1=0$. Observe that this expression is indeed well-defined, because from the above considerations
one sees that $\nabla\left|\de\rho\right|^2\in L^\infty_T(L^2)$.

At this point, let us define the energy functional $E(t)$, for $t\in[0,T]$, as the quantity
\[
 E(t)\,:=\,\left\|\sqrt{\rho_1(t)}\,\de u(t)\right\|_{L^2}^2\,+\,\left\|\de\rho(t)\right\|_{L^2}^2\,.
\]
Recall that $E(0)\,=\,0$ by assumption.
Summing up equalities \eqref{est:du-L^2} and \eqref{est:drho-L^2}, we infer that
\begin{equation} \label{est:E_prelim}
 \frac{\dd}{\dt}E(t)\,\lesssim\,\left|\int_{\R^2}\dfrac{\de\rho}{\rho_2}\,\nabla\Pi_2\cdot\de u\,\dx\right|\,+\,
\left|\int_{\R^2}\rho_1\,\de u\cdot\nabla u_2\cdot\de u\,\dx\right|\,+\,\left|\int_{\R^2}\de u\cdot\nabla\rho_2\,\de\rho\,\dx\right|\,.
\end{equation}
To conclude the proof, we are now going to estimate one by one all the terms on the right-hand side of the previous relation.

Using \eqref{apr_est:dens} and 
the fact that $\nabla\rho_2\in L^\infty\big([0,T]\times\R^2\big)$, 
we easily get the following inequality:
\begin{align} \label{est:T_3}
\left|\int_{\R^2}\de u\cdot\nabla\rho_2\,\de\rho\,\dx\right|\,&\lesssim\,\left\|\nabla\rho_2\right\| _{L^\infty}\,E(t)\,.
\end{align}
Controlling the first and second terms in \eqref{est:E_prelim}, instead, requires more care,
as, under our assumptions, we do not dispose of a $L^\infty$ bound for the gradient of the velocities, nor for the pressure gradient.
Therefore, we will resort to the original argument used by Yudovich to prove his theorem (see also Chapter 8 of \cite{Maj-Bert} for details).

To begin with, we focus on the first term on the right-hand side of \eqref{est:E_prelim}.
Thanks to estimate \eqref{apr_est:Pi-p}, we can write
\begin{align*}
\left|\int_{\R^2}\dfrac{\de\rho}{\rho_2}\,\nabla\Pi_2\cdot\de u\,\dx\right|\,&\lesssim\,\left\|\nabla\Pi_2\right\|_{L^p}\,\left\|\de\rho\right\|_{L^q}\,
\left\|\sqrt{\rho_1}\,\de u\right\|_{L^2}\,,
\end{align*}
where $p$ and $q$, both belonging to $[2,+\infty[\,$, have been taken, respectively, very large and very close to $2$, so that
$1/p\,+\,1/q\,=\,1/2$. An interpolation inequality, together with estimate \eqref{apr_est:dens} for both $\rho_1$ and $\rho_2$, then yields
\begin{align*}
\left|\int_{\R^2}\dfrac{\de\rho}{\rho_2}\,\nabla\Pi_2\cdot\de u\,\dx\right|\,&\lesssim\,\left\|\nabla\Pi_2\right\|_{L^p}\,\left\|\de\rho\right\|_{L^2}^{\theta}\,
\left\|\sqrt{\rho_1}\,\de u\right\|_{L^2}\,,
\end{align*}
where we have defined $\theta\, =\, 2/q\,=\,1\,-\,2/p$. Applying inequality \eqref{apr_est:Pi-p}, we finally deduce that
\begin{align} \label{est:T_1}
 \left|\int_{\R^2}\dfrac{\de\rho}{\rho_2}\,\nabla\Pi_2\cdot\de u\,\dx\right|\,&\lesssim\,\sqrt{p}\,\Big(E(t)\Big)^{1-\frac1p}\,.
\end{align}

Treating the second term appearing in the right-hand side of \eqref{est:E_prelim} relies on similar arguments.
We recall that $\nabla u_2$ satisfies \eqref{eq:integrab_Du}, namely $\nabla u_2\in L^\infty_T(L^p)$ for all $p_0\leq p<+\infty$.
On the other hand, $\de u$ belongs to $L^2\cap L^\infty$, as each $u_j$ does. Therefore, taken $p>p_0$ large and $q$ such that $1/p\, +\, 1/q\, =\, 1/2$ as before,
we can bound
\begin{align*}
\left|\int_{\R^2}\rho_1\,\de u\cdot\nabla u_2\cdot\de u\,\dx\right|\,\leq\,\left\|\nabla u_2\right\|_{L^p}\,\left\|\sqrt{\rho_1}\,\de u\right\|_{L^q}\,
\left\|\sqrt{\rho_1}\,\de u\right\|_{L^2}\,.
\end{align*}
Using the same interpolation argument as above, the bound \eqref{apr_est:u-inf} for $u_1$ and $u_2$ in $L^\infty$ and the Calder\'on-Zygmund estimate
\[
\forall\,p\in\,]1,+\infty[\,,\qquad\qquad \left\|\nabla u\right\|_{L^p}\,\lesssim\,\frac{p^2}{p-1}\,\left\|\o\right\|_{L^p}\,,
\]
in turn we find, thanks to \eqref{apr_est:omega}, the inequality
\begin{align} \label{est:T_2}
\left|\int_{\R^2}\rho_1\,\de u\cdot\nabla u_2\cdot\de u\,\dx\right|\,\leq\,\frac{p^2}{p-1}\,\Big(E(t)\Big)^{1-\frac1p}\,.
\end{align}

It is now time to gather estimates \eqref{est:T_3}, \eqref{est:T_1} and \eqref{est:T_2} together and put them into \eqref{est:E_prelim}: as $p>p_0$
is very large, we deduce that
\begin{align*}
 \frac{\dd}{\dt}E(t)\,\leq\,C_1\,E(t)\,+\,C_2\,p\,\Big(E(t)\Big)^{1-\frac1p}\,,
\end{align*}
for any $t\in[0,T]$, for two universal positive constants $C_1$ and $C_2$, possibly depending on the norms of the initial datum $\big(\rho_0,u_0\big)$
and on the quantity $M_0$ defined in \eqref{def:K_0}.

This bound having been obtained, the rest of the argument is very similar to \tsl{e.g.} the proof of Theorem 8.2 of \cite{Maj-Bert}:
the presence of the linear term $C_1\,E(t)$ on the right-hand
side of the previous estimate represents only a minor difference. Indeed, we can define the time $T_0$ as
\[
 T_0\,:=\,\sup\left\{t\in[0,T]\;\Big|\quad E(t)\,\leq\,\frac{1}{2}\right\}\,.
\]
Observe that $E(0)=0$; in addition, an inspection of the equation for $\de\rho$ in system \eqref{eq:delta-eq} and Lemma \ref{l:d_tu}
imply that $E$ is a continuous quantity in time. Therefore, $T_0$ is well-defined and one has $T_0>0$.
Then, on the (possibly smaller, as $T_0\leq T$ in principle) time interval $[0,T_0]$, we have
\begin{equation} \label{est:dE-dt}
 \frac{\dd}{\dt}E(t)\,\leq\,K\,p\,\Big(E(t)\Big)^{1-\frac1p}\,,
\end{equation}
where we have set $K\,:=\,C_1+C_2$. Despite the previous inequality has not a unique solution for $p$ fixed, it does possess a maximal solution
\[
 E_{\max}(t)\,=\,\Big(K\,t\Big)^p\,,
\]
so that any other solution of \eqref{est:dE-dt} satisfies $E(t)\,\leq\,E_{\max}(t)$ for all $t\geq0$. See also Lemma B.2 of \cite{A-B-C-DL-G-J-K}
for more details.
At this point, we define the time $T_1$ such that
\[
 T_1\,:=\,\min\left\{T_0\,,\,\frac{1}{2K}\right\}\,.
\]
Then, from \eqref{est:dE-dt} and the previous considerations we deduce the inequality 
\[
\sup_{t\in[0,T_1]} E(t)\,\leq\,\Big(K\,T_1\Big)^p\,\leq\,2^{-p}\,.
\]
At this point, sending $p\to+\infty$ we get $E\equiv0$ on $[0,T_1]$. From this, it is then a routine matter to deduce that
$E\equiv 0$ on the whole time interval $[0,T]$, which in turn implies the sought uniqueness of solutions.

This completes the proof of Theorem \ref{th:yudovich}.


%


\end{document}